\documentclass[oneside,10pt,reqno]{amsart}
\usepackage{amssymb,amsmath,amsthm,bbm,enumerate,hyperref}
\usepackage[shortlabels]{enumitem}

\theoremstyle{definition}
\newtheorem{definition}{Definition}
\theoremstyle{theorem}
\newtheorem{proposition}[definition]{Proposition}
\newtheorem{lemma}[definition]{Lemma}
\newtheorem{theorem}[definition]{Theorem}
\newtheorem{corollary}[definition]{Corollary}

\numberwithin{equation}{section}
\numberwithin{definition}{section}
\theoremstyle{remark}
\newtheorem{remark}[definition]{Remark}
\newtheorem{example}[definition]{Example}

\def\PP{\mathsf{P}}
\def\QQ{\mathsf{Q}}
\def\EE{\mathsf{E}}
\def\FF{\mathcal{F}}
\def\GG{\mathcal{G}}
\def\AA{\mathcal{A}}
\def\BB{\mathcal{B}}
\def\LL{\mathcal{L}}	
\def\ee{\mathsf{e}}

\def\dd{\mathrm{d}}
\def\pp{\tilde{p}}
\def\rr{\tilde{r}}
\def\II0{\mathcal{I}^q_0}
\def\IIinf{\mathcal{I}^q_\infty}

\def\qbar{{\bar{q}}}

\makeatletter
\@namedef{subjclassname@2020}{%
  \textup{2020} Mathematics Subject Classification}
\makeatother
\allowdisplaybreaks
\begin{document}
\title[Harmonic functions for killed Markov branching processes with immigration and culling]{Some harmonic functions for killed Markov branching processes with immigration and culling}

\begin{abstract}
For  a continuous-time Bienaym\'e-Galton-Watson process, $X$, with immigration and culling,  $0$ as an absorbing state, call $X^q$ the process that results from  killing $X$ at rate $q\in (0,\infty)$, followed by  stopping it on extinction or explosion. Then an explicit identification of the relevant harmonic functions of $X^q$ allows to determine the Laplace transforms (at argument $q$) of the first passage times downwards and of the explosion time for $X$. Strictly speaking, this is accomplished only when the killing rate $q$ is sufficiently large (but always when the branching mechanism is not supercritical or  if there is no culling). In particular, taking the limit $q\downarrow 0$ (whenever possible) yields the passage downwards and explosion probabilities for $X$. A number of other  consequences of these results are presented. 
\end{abstract}
\author{Matija Vidmar}
\address{Department of Mathematics, Faculty of Mathematics and Physics, University of Ljubljana}
\email{matija.vidmar@fmf.uni-lj.si}
\thanks{Financial support from the Slovenian Research Agency is acknowledged (programme No. P1-0402).
}
\maketitle

\section{Introduction}
Branching processes are ubiquitous in the modeling of many real-world phenomena that are subject to the laws of chance,  ``avalanches, networks, earthquakes, family
		names, populations of bacteria and cells,
		nuclear reactions, cultural evolution and neuronal
		avalanches'' \cite[first paragraph]{field-theory}. They represent a fundamental family of (types of) stochastic processes in discrete or continuous time or space, that has received considerable attention in the mathematical literature \cite{athreya,harris,asmussen,lambert,li2010measure}. Many fine results are available, especially concerning their asymptotic behavior, while much less is known about the laws of the quantities attached to the exits of branching processes from (even semi-infinite) intervals.
\subsection{Specification of the class of processes under consideration}
Let $X=(X_t)_{t\in [0,\infty)}$ be a  continuous-time Bienaym\'e-Galton-Watson process (ctBGWp) \cite[Chapter~3]{athreya} \cite[Chapter~V]{harris} \cite[Chapter~IV]{asmussen} \cite[Section~1.2]{lambert} under the probabilities $\PP=(\PP_x)_{x\in \mathbb{N}_0}$, defined on a base measurable space $(\Omega,\GG)$.

Recall that, informally, for $x\in \mathbb{N}_0$, under $\PP_x$, $X$ measures the size of a population in which all individuals die, and give birth to offspring at the times of their death, all at the same rate, $\lambda\in (0,\infty)$, and  with the same number-of-offspring distribution, $p=(p_k)_{k\in \mathbb{N}_0}$, independently of each other, $x$ being the initial number of individuals. Put differently: starting with a population of size $x$, each individual, independently of the others, stays alive for a random amount of time distributed exponentially with rate $\lambda$; upon its death, for $k\in \mathbb{N}_0$,  $p_k$ is the probability of it giving birth to precisely $k$ offspring. In particular, $p_0$ is the probability of dying without producing progeny.

CtBGWp, being continuous-time Markov chains (ctMc) \cite{norris1998markov,chung1967markov},  are arguably the simplest  branching processes in continuous time.  They are also the closest continuous-time analogues of the basic (discrete-time) Bienaym\'e-Galton-Watson processes (BGWp) \cite[Chapter~I]{athreya}. Though, the continuous-discrete time correspondence is not without limitation: the ctBGWp are skip-free downwards (i.e. a.s. do not skip any level as they attain new minima), while the BGWp are not. Accordingly a ctBGWp is not simply a subordination of a BGWp, and only some BGWp are embeddable in a ctBGWp \cite[Section~III.12]{athreya}. Still, every discrete-time skeleton of a ctBGWp is a BGWp \cite[Section~III.6]{athreya}. Scaling limits of (ct)BGWp lead to continuous-state branching processes (csbp) \cite[Chapter~3]{li2010measure}. 

In this paper we will be interested in the \emph{continuous-time} \emph{discrete-space} process $X$ only. That being so, we add to the ``branching constellation'' described above immigration, \emph{as well as} -- culling. More precisely, independently of the branching, at some rate $\mu\in [0,\infty)$, \emph{until the process dies out or explodes}, we either immigrate a certain number of individuals, or \emph{cull}, but at any given point in time at most one individual, according to the probability distribution function $r=(r_k)_{k\in \{-1\}\cup \mathbb{N}}$, culling of one individual occurring with probability $r_{-1}$, immigration of $k\in \mathbb{N}$ individuals occurring with probability $r_{k}$.  

The presence of culling precludes the ``branching property'' and we do not have available a semi-explicit representation of the semigroup as in the classical branching (with immigration) case. This makes the study more involved. On the other hand, the addition of culling and immigration does preserve the downwards skip-free property, which is what in the end ``saves the day''. \label{branching-property}

We have  the special cases: (i) $\mu=0$ --- no immigration/culling, i.e. pure branching; (ii) $r_{-1}=0$ --- no culling, i.e. branching with immigration (stopped on extinction). Besides, if, ceteris paribus, we were to allow $\lambda=0$ (but insist that $\mu r_{-1}>0$), then $X$ would simply be a, stopped on hitting $0$, homogeneous-Poisson-process--subordinated, integer-valued left-continuous random walk. For the latter we refer to \cite{vidmar2013fluctuation} and retain the standing assumption $\lambda>0$ (except where otherwise indicated). 
To avoid some complications/trivial considerations we will also assume throughout that $p_0>0$ and that $X$ does not have a.s. nonincreasing paths (so either $p_0+p_1<1$, or else $\mu>0$ and $r_{-1}<1$).

\subsection{Overview of results}
Let now $T_0^-$ be the first hitting time of $0$ by the process $X$ and let $q\in (0,\infty)$. Set: $\varphi:=\PP_1(T_0^-<\infty)\in (0,1]$, wherein we take $\mu=0$, leaving $(p,\lambda)$ unchanged (in particular $\varphi=1$ if $p_0=1$);  $\phi_q:=\PP_1[e^{-qT_0^-};T_0^-<\infty]\in [0,1)$, wherein we take $\lambda=0$, leaving $(r,\mu)$ unchanged (so $\phi_q=0$ if there is no culling).\footnote{We write expectations in the parlance of the theory of Markov processes: $\PP_x[F;A]=\EE_{\PP_x}[F\mathbbm{1}_A]$, whenever it is defined.} The quantities $\varphi$ and $\phi_q$ can be specified explicitly in terms of the p.g.f. of $p$ and $r$, respectively (see below). 

Then, under the condition $\phi_q\leq \varphi$ (resp. $\phi_q<\varphi$ and $X$ can explode) this paper  will provide an explicit expression for a non-zero, bounded, vanishing at infinity (resp. \emph{and} an explicit expression for a bounded, \emph{not} vanishing at infinity), harmonic  function of the process $X$ that is killed at rate $q$ and stopped on becoming extinct or exploding. The expressions will be ``explicit'' in the sense that they will be expressed directly in terms of $q$, $\lambda$, $\mu$ and the p.g.f. of $p$ and $r$. See Theorem~\ref{thm:skip-free}. 

As a consequence of the preceding we will obtain the Laplace transforms of the first passage times downwards of $X$ (resp. \emph{and} of the explosion time of $X$) at all arguments $q$ for which $\phi_q\leq \varphi$ (resp. $\phi_q< \varphi$), Theorem~\ref{corollary:laplace}, together with some immediate related corollaries to do with the following, among others: the computation of the means of first passage and explosion times (Corollaries~\ref{corollary:first-passage-mean} and~\ref{corollary:explosion}); the conditioning of $X$ on extinction before an independent exponential random time has elapsed (Corollary~\ref{corollary:condition}); the factorization of $X$ at the minimum up to an independent exponential random time (Corollary~\ref{corollary:at-min}). In particular, taking the limit $q\downarrow 0$ will yield expressions for the passage downwards (hence extinction) and explosion (before passage downwards) probabilities of $X$, provided $\lim_{q\downarrow 0}\phi_q\leq \varphi$. A change of measure allows one to include the cumulative-lifetime-to-date process $\int_0^\cdot X_t\dd t$ at first passage, Corollary~\ref{corollary:avalanche-size}. 

 Besides the generality (immigration, culling, explosions), the main appeal of the present work is the level of explicitness that one can attach to the harmonic functions. The fact that we also obtain the harmonic function corresponding to explosions is  a further distinguishing element of our analysis. A drawback is that, by and large, we are not able to escape the ``notorious vice of applied probabilists to present their results hidden
 behind one or more Laplace transforms'' \cite[p.~85]{homecoming}.
\subsection{Connections to existing literature;  indication of some related problems that are left open}\label{subsection:lit}
Our deliberations contribute mainly to the literature on first passage times, which is a venerable topic of the theory of real stochastic processes, with considerable practical relevance \cite{redner}. 

Specifically as it concerns first passage theory of branching processes, our results complement the following two papers. On the one hand, \cite{Avram2019}, that considered the first passage downwards problem of ctBGWp, but did not include immigration/culling, and provided an expression for the relevant excessive function  in terms of the transition semigroup only (except in special cases), cf. Remark~\ref{remark:grey}. On the other hand, \cite{ma}, which considered the analogous first passage problem for csbp with immigration (cbi), providing in fact for this class an explicit formula \cite[Eq.~(11)]{ma} that even includes the cumulative-lifetime-to-date (a.k.a. total progeny, avalanche size) process at first passage, besides the first passage time itself. This formula, expressed in terms of the Laplace exponents of the underlying branching/immigration mechanisms, is very much akin to ours \eqref{eq:with-avalanche-size}, cf. Remark~\ref{rmk:akin}. This is not surprising since cbi are scaling limits of ctBGWp  with immigration. By a similar token, culling, just like immigration, should also allow for an analogue in the continuous state-space setting, cf. Remark~\ref{remark:lamperti}. This, together with the treatment of the time of explosion for csbp with immigration and culling, is left to be pursued in future work. Further connections with \cite{Avram2019,ma} are provided in the main body of the exposition, as appropriate. 

We may draw a parallel with the theory of scale functions of upwards or downwards skip-free Markov processes. Indeed, for processes with stationary independent increments (psii) this skip-free property contributes to a significant simplification of their fluctuation theory, with the fundamental exit problems being then parsimoniously expressible in terms of  a collection of so-called scale functions \cite{kkr,omega,vidmar-avram,vidmar2013fluctuation}. Outside of the psii class the situation becomes more complex \cite{maps,pssmp}, except for special cases, e.g. diffusions that have continuous sample paths \cite[Section V.46]{rogers2000diffusions}. The harmonic function  $\Phi_q$ defined below is nothing but the simplest scale function attributable to ctBGWp on account of their downwards skip-free property, solving the first-passage downwards problem. Characterizing efficiently the whole suite of scale functions of ctBGWp or csbp with immigration and/or culling, which would solve the two-sided exit problem for these classes of processes, remains an open task. 

This paper relates  in general to the literature on ctBGWp and their generalizations. Of the more recent results we may mention \cite{field-theory} that contains an interesting exploration of the  asymptotically universal behaviour of ctBGWp near criticality, in particular with regard to the extinction time and total progeny (avalanche duration and size, respectively, in the terminology of \cite{field-theory}). Also \cite{li-infinite}, which deals with the asymptotics of ctBGWp with (possibly infinite) immigration and allowing also infinite offspring. 

 From another point of view, our results fall under  the study of excessive functions of Markov processes, and their connections to Martin boundaries, viz. the ways in which a  Markov process exits its state-space, see e.g. \cite[Chapter~7]{woess}, \cite[passim]{dynkin}, and \cite{choi} for upwards/downwards skip-free Markov chains in particular. 
\subsection{Article structure}
We give the precise results concerning harmonic functions in Section~\ref{sec:skip-free}, after establishing some preliminaries in Section~\ref{sec:preliminaries}. Then Section~\ref{section:laplace} delivers the promised Laplace transforms, while Section~\ref{applications} is devoted to further consequences.


\section{Preliminaries}\label{sec:preliminaries}
Without loss of generality we may, and do assume $p_1=0$ (the general case requires only changing the rate at which the individuals are dying/reproducing to $\lambda(1-p_1)$ and conditioning the offspring distribution on $\mathbb{N}_0\backslash \{1\}$). On the other hand, we put $r_0:=0$.  We refer to $p$ as the branching, and to $r$ as the immigration/culling mechanism; also $\lambda$ (resp. $\mu$) is called the reproduction (resp. immigration/culling) rate. 

Formally, the system $(X,\PP)$ is a minimal c\`adl\`ag  ctMc with state space $\mathbb{N}_0$, lifetime $\zeta$ and cemetery $\infty$ ($X$ is set equal to $\infty$ on $[\zeta,\infty)$). Its conservative generator matrix $Q$ on $\mathbb{N}_0$ satisfies\footnote{By the qualifier ``conservative'' we mean merely that the sum of each row of $Q$ is zero.}: $Q_{00}=0$, while  for $n\in \mathbb{N}$, $-Q_{nn}=\lambda n+\mu$ and  $Q_{n(n+m)}=p_{m+1}\lambda+r_m\mu$ for further $m\in \mathbb{N}\cup \{-1\}$. The generator $Q$ is irreducible on restriction to $\mathbb{N}$ and all the states in $\mathbb{N}$ are transient.

A couple of the particulars agreed on thus far,  that are not entirely innocuous, are  a little hidden from sight; let us stress them for the reader's benefit once more. To wit: $0$ is an absorbing state, even though under immigration without culling it might not have been rendered as such; $\lambda p_0>0,$ even though also the case $p_0=0$ is of interest when $r_{-1}>0$ (while  we have already commented on what happens for $\lambda=0$); we preclude a.s. nonincreasing paths, i.e.  we assume that
$$\text{$\mu r_k>0$ for some $k\in \mathbb{N}$ or $p_0<1$.}$$ 

\begin{remark}\label{remark:lamperti}
	In the spirit of \cite{caballero2013} we may note as follows. Under the probabilities $(\QQ_x)_{x\in \mathbb{N}_0}$ take independent random elements: two independent downwards skip-free random walks on $\mathbb{Z}$, $W^1$ and $W^2$, with jump probabilities $(p_{k+1})_{k\in \{-1\}\cup \mathbb{N}}$ and $r$, respectively; two independent homogeneous Poisson processes, $N^1$ and $N^2$, of intensities $\lambda$ and $\mu$, respectively. We insist that $W^1_0=x$ and $W^2_0=N^1_0=N^2_0=0$ a.s.-$\QQ_x$ for all $x\in \mathbb{N}_0$  (only nonnegative starting values of $W^1$ are of interest).  Set $Y^i:=W_{N^i}^i$ for $i\in \{1,2\}$. Then a.s. there is a unique  $\mathbb{N}_0$-valued c\`adl\`ag path $Z=(Z_t)_{t\in [0,\xi)}$  with lifetime $\xi$,  $0$ as an absorbing state, such that, with $\tau_0$ the hitting time of  $0$ by $Z$,  $$\left(Z_t=Y^1_{\int_0^t Z_s\dd s}+Y_{t\land \tau_0}^2\text{ for }t\in [0,\xi)\right)\text{ and }\left(\lim_{ \xi-}Z=\infty\text{ if }\xi<\infty\right).$$ The process $Z$  under the probabilities $(\QQ_x)_{x\in \mathbb{N}_0}$ is  seen to be a realization of $(X,\PP)$ as described above. Thus a ctBGWp with immigration and culling comes from a path transformation of two downwards skip-free random walks. Replacing $Y^1$ and $Y^2$ with independent spectrally positive L\'evy processes is then likely to represent one possible avenue into defining a csbp with immigration and culling, but we will not pursue this here.
\end{remark}
We set next, for $z\in (0,1]$, 
$$
\pp(z):=\sum_{k=0}^\infty p_kz^{k}\text{ and }\, 
\rr(z):=\sum_{k=-1}^\infty r_kz^k,
$$
the p.g.f. of $p$ and $r$, respectively.  
Then $\varphi$ will denote the smallest root  of $\pp(z)=z$ in $z\in (0,1]$, there being at most one more in addition to $1$. The case $\pp'(\varphi-)=1$ occurs iff $\pp'(1-)=1$, in which case also $\varphi=1$, and this corresponds to the critical branching mechanism. On the other hand,  when $\pp'(1-)$ is $<1$, and so $\varphi=1$ (resp. is $>1$, and so $\varphi<1$) the branching mechanism is called subcritical (resp. supercritical).  We recall further that the condition (which can only obtain when $\varphi<1$)
\begin{equation}\label{eq:ee}
\int^1_v\frac{\dd z}{z-\pp(z)}<\infty\text{ for some (then all) }v\in (\varphi,1)
\end{equation}
is equivalent to (recall $\zeta$ is the explosion time)
\begin{equation}\label{eq:e}
\text{$\PP_x(\zeta<\infty)>0$ for some (equivalently, all) $x\in \mathbb{N}$,}\tag{E}
\end{equation}
i.e. to the explosivity of the chain $X$ \cite[Theorem V.9.1]{harris} -- the presence of the immigration and culling is without effect on the (non-)explosiveness of $X$, which seems plain (because immigration/culling occur at a constant rate), but let us give a formal proof at once. 

\noindent $\diamond$ 
 Assume $\mu>0$. We insist that $X$ has been got from a direct construction using individuals' lifetimes, their respective number of offspring random variables and an immigration/culling process. We may do this because we are only interested in a distributional statement. In particular, the probabilities $\PP_x$, $x\in \mathbb{N}_0$, are rich enough to support knowledge of  $\gamma:=\inf\{t\in [0,\zeta):\text{in the construction } X\text{ sees an immigration/culling at time }t\}$ ($\inf\emptyset:=\infty$; we have stressed the presence of the construction because $\gamma$ is actually not a random time of the process $X$ itself); $\{\zeta<\gamma\}=\{\zeta<\gamma= \infty\}$ is the event that there is explosion before the first immigration/culling has taken place. Two observations are key. First, $X$ has on $[0,\gamma)$ the law of $\tilde{X}$ on $[0,\tilde{\gamma})$, where, under the probabilities $\tilde{\PP}_x$, $x\in \mathbb{N}_0$, we have that $\tilde{X}$ is a realization of $X$ with pure branching (so with, ceteris paribus, $\mu=0$) and $\tilde{\gamma}:=\mathsf{e}\mathbbm{1}_{\{\mathsf{e}<\tilde{\zeta}\}}+\infty\mathbbm{1}_{\{\tilde{\zeta}\leq \mathsf{e}\}}$, $\mathsf{e}$ an  exponential random time of mean $\mu^{-1}$ independent of $\tilde{X}$, $\tilde{\zeta}$ the explosion time of $\tilde{X}$. In particular, $\{\tilde\zeta<\tilde\gamma\}=\{\tilde\zeta\leq \mathsf{e}\}$, hence $\PP_x(\zeta< \gamma)=\tilde\PP_x(\tilde\zeta< \tilde\gamma)=\tilde\PP_x(\tilde\zeta\leq \mathsf{e})$ for all $x\in \mathbb{N}$. Second, again for any given $x\in \mathbb{N}$, $\tilde\PP_x(\tilde{\zeta}<\infty)>0$ iff $\tilde \PP_x(\tilde{\zeta}\leq \mathsf{e})>0$ (trivial, because of independence). With these observations in hand, fix $x\in \mathbb{N}$. If $\tilde \PP_x(\tilde{\zeta}<\infty)>0$,  then from the preceding we get at once that $\PP_x(\zeta<\infty)>0$. Conversely, if $\PP_x(\zeta<\infty)>0$, then we must have $\PP_x(\zeta< \gamma)>0$ because otherwise $\tilde \PP_x(\tilde{\zeta}\leq \mathsf{e})=0$, hence $\tilde\PP_x(\tilde{\zeta}<\infty)=0$, so $\tilde\PP_y(\tilde{\zeta}<\infty)=0$ and therefore $\PP_y(\zeta<\gamma)=0$ for all $y\in \mathbb{N}$, which by induction, the strong Markov property, and the strong law of large numbers (applied to the times in-between culling/immigration events) renders $\PP_x(\zeta<\infty)=0$, a contradiction. Now $\PP_x(\zeta<\infty)>0$ is seen at once to imply $\tilde \PP_x(\tilde{\zeta}<\infty)>0$. \qed
\label{proof-pathwise}

In passing we may recall for the reader (we will not need it)  that by a result of  \cite[Corollary~2]{doney} condition \eqref{eq:e} is also equivalent to 
$$\sum_{n=1}^\infty\left(n\sum_{k=0}^n\sum_{l=k+1}^\infty p_l\right)^{-1}<\infty.$$ In particular the chain is always explosive if  $\sum_{l=k}^\infty p_l\sim k^{-\alpha}L(k)$ as $k\to \infty$ for a slowly varying $L$ and $\alpha\in (0,1)$ \cite{grey}; the situation corresponding to \eqref{eq:e} is non-vacuous. 

In terms of first passage quantities --- that will be of fundamental importance --- for $a\in \mathbb{N}_0$ we denote by  $T_a^-:=\inf \{t\in [0,\zeta): X_t\leq a\}$ [$\inf\emptyset=\infty$] the first passage time below  the level $a$. Because all the states in $\mathbb{N}$ are transient and because $0$ is absorbing, a.s. $\Omega$ is equal to the disjoint union of $\{T_0^-<\infty\}$, $\{\zeta=\infty,\lim_\infty X=\infty\}$ and  $\{\zeta<\infty,\lim_{\zeta-}X=\infty\}$. If $\mu=0$, then for all $x\in \mathbb{N}_0$,  $\varphi^x=\PP_x(T_0^-<\infty)$ is the extinction probability of $X$ under $\PP_x$.\label{transience}

When the qualifiers a.s., independent, martingale etc. shall appear  below without specification of a probability measure, they are asserted under $\PP_x$ for all $x\in \mathbb{N}_0$. We grant ourselves access to a $(0,\infty)$-valued random exponential time of rate $1$, $\ee_1$, independent of $X$; then we set $\ee_q:=\ee_1/q$ for $q\in [0,\infty)\backslash \{1\}$. We will be adding killing to $X$ at the times $\ee_q$, $q\in (0,\infty)$. The cemetery for this killing will be $-\infty$. For any function $f$ not defined at $-\infty$ we understand $f(-\infty):=0$ (but leave $f$ with its domain such as it is).

\section{Harmonic functions for the killed process}\label{sec:skip-free}

Throughout this section let $$q\in (0,\infty)\text{ and define the process }X^q:=X\mathbbm{1}_{\{\cdot \land T_0^-\land \zeta<\ee_q\}}+(-\infty)\mathbbm{1}_{\{\cdot \land T_0^-\land \zeta\geq \ee_q\}},$$ which is $X$, \emph{first} killed and sent to the cemetery $-\infty$ at the time $\ee_q$, and \emph{then} stopped on hitting $\{0,\infty\}$. We stress: to go from the law of $X$ to the law of  $X^q$ one adds killing (at rate $q$)  in the non-absorbing states \emph{only}.

The following proposition characterizes the bounded harmonic functions of $X^q$ (with a well-defined limit at infinity). 

\begin{proposition}\label{lemma:mtg}
	Let $f:\mathbb{N}_0\to \mathbb{R}$ be bounded and assume that 
\begin{equation}\label{eq:condition} \text{the limit $f(\infty):=\lim_\infty f$ exists whenever \eqref{eq:e} prevails.}
\end{equation}
	Define $W^q_f:=(e^{-q(t\land T_0^-\land \zeta)}f(X_{t}))_{t\in [0,\infty)}$. The following statements are equivalent. 
	\begin{enumerate}[(i)]
		\item\label{lemma:mtg:1} $f$ is harmonic for $X^q$,  i.e. it renders $f(X^q)$ a martingale in the natural filtration of $X^q$. 
		\item\label{lemma:mtg:1'} $f(X^q)$ is a martingale in the smallest filtration that makes $X$ adapted and $\ee_q$ a stopping time.
				\item\label{lemma:mtg:2'} The process $W^q_f$ is a martingale in the natural filtration of $X$.
		\item\label{lemma:mtg:3} For all $x\in \mathbb{N}$, 
\begin{equation}\label{eq:harmonic}
(q+\lambda x+\mu)f(x)=\lambda x\sum_{k=0}^\infty p_kf(x+k-1)+\mu\sum_{k=-1}^\infty r_kf(x+k).
\end{equation}
\end{enumerate}
\end{proposition}
Such a statement can certainly be considered valid ``as part of folklore''. Nevertheless we provide a proof, since we allow for explosions, and in principle subtleties could have been overlooked.
\begin{proof}
Notation-wise, let $(J_n)_{n\in \mathbb{N}}$ be the sequence of jump times of $X$ (if there are only finitely many, then the remaining ones are set equal to $\infty$). Set $J_0:=0$, and define $H_n:=X_{J_n}$ for $n\in \mathbb{N}_0$ (here $X_\infty:=0$, of course). The process $H=(H_n)_{n\in \mathbb{N}_0}$ is the jump chain of $X$. We denote by $\tau_0^-$  the first time $H$ enters $\{0\}$.

Suppose first that \ref{lemma:mtg:1} or \ref{lemma:mtg:2'} holds; we show \ref{lemma:mtg:3}. Indeed, by optional sampling, $\PP_x[f(X^q_{J_1})]=\PP_x[f(X^q_{J_1\land\ee_q})]=f(x)$ ($J_1\land\ee_q$ is a stopping time of $X^q$ and $X^q_{J_1}=X^q_{J_1\land\ee_q}$) or $\PP_x[W^q_f(J_1)]=f(x)$, and \eqref{eq:harmonic}
 follows from the structural characterization of ctMc. 

Now assume \ref{lemma:mtg:3} holds; we prove \ref{lemma:mtg:2'}. This could be achieved using results from general Markov process theory \cite[Proposition~4.1.7]{ethier}, but in this context it is probably easier to follow a direct approach. To wit, using \eqref{eq:harmonic} one establishes by induction that the discrete-time process $(f(H_n)\prod_{k=1}^{n\land \tau_0^-}\frac{\mu+\lambda H_{J_{k-1}}}{q+\mu+\lambda H_{J_{k-1}}})_{n\in \mathbb{N}_0}$ is a martingale in the natural filtration of the jump chain of $X$. Consequently, from the structural characterization of ctMc, by ``integrating out'' the normalized holding periods, one sees that for each $n\in \mathbb{N}$, the process $W^q_f$, stopped at $J_{n}$, viz. the process $(e^{-q(t\land T_0^-\land J_n)}f(X_{t\land J_n}))_{t\in [0,\infty)}$, is a martingale in the natural filtration of $X$. Letting $n\to\infty$ in the preceding  yields the martingale property of $W^q_f$ (one uses \eqref{eq:condition} here). 

\ref{lemma:mtg:1'}, which in turn trivially implies \ref{lemma:mtg:1} (since the natural filtration of $X^q$  is included in the smallest filtration that makes $X$ adapted and $\ee_q$ a stopping time), is a mere rewriting of \ref{lemma:mtg:2'}, using the independence of $X$ and $\ee_q$ (and the Markov property of $X$ coupled with the memoryless property of $\ee_q$). 
\end{proof}

\begin{definition}
	 We denote by  $\mathcal{I}^q$ the space of bounded $f:\mathbb{N}_0\to \mathbb{R}$ that satisfy \eqref{eq:condition} and \eqref{eq:harmonic}. We put $\II0:=\{f\in\mathcal{I}^q:\lim_\infty f=0\}$ and $\IIinf=\{f\in\mathcal{I}^q:f(0)=0\}$.
\end{definition}

Let us provide a ``landscape view'' of $\mathcal{I}^q$.

\begin{proposition}\label{proposition:uniqueness}
$\mathcal{I}^q$ is a vector space of dimension one (resp. two) when \eqref{eq:e} fails (resp. prevails). $\II0$ is a vector subspace of $\mathcal{I}^q$ of dimension one. If \eqref{eq:e} prevails, then $\IIinf$ is also a vector subspace of $\mathcal{I}^q$ of dimension one. Let $f\in \II0\backslash\{0\}$. Then $f\ne 0$ everywhere on $\mathbb{N}_0$, $f$ is strictly decreasing, and for $\{x,a\}\subset \mathbb{N}_0$, 
\begin{equation*}
\PP_x[e^{-q T_a^-};T_a^-<\infty]=\frac{f(x)}{f(a)},\quad a\leq x.
\end{equation*}
If further $g\in \mathcal{I}^q\backslash \II0$, then \eqref{eq:e} prevails and for $\{x,a\}\subset \mathbb{N}_0$,
\begin{equation*}
\PP_x[e^{-q \zeta};\zeta<T_a^-]=\frac{g(x)}{g(\infty)}-\frac{g(a)}{g(\infty)}\frac{f(x)}{f(a)},\quad a\leq x.
\end{equation*}
If even $g\in \IIinf\backslash\{0\}$, then $g\ne 0$ on $\mathbb{N}$, $g$ is strictly increasing, and for $x\in \mathbb{N}_0$,
\begin{equation*}
\PP_x[e^{-q \zeta};\zeta<\infty]=\frac{g(x)}{g(\infty)}.
\end{equation*}
\end{proposition}
Informally speaking, the preceding statement is connected to the fact that $X^q$ can exit $\mathbb{N}$ in  one or two ways, according as explosions cannot or can happen (it can always exit $\mathbb{N}$ by hitting $0$). As alluded to in the Introduction, this is part of a much grander story of Martin boundaries, however in the present case the proof is elementary. 
\begin{proof}
Let $h\in \mathcal{I}^q$. 	Applying optional sampling to the martingale $W^q_h$ from Proposition~\ref{lemma:mtg} we find that for all $x\in \mathbb{N}_0$,
	$$h(x)=h(0)\PP_x[e^{-q T_0^-};T_0^-<\infty]+h(\infty)\PP_x[e^{-q\zeta};\zeta<\infty],$$
and more generally, for further $a\in \mathbb{N}_0$, 
	$$h(x)=h(a)\PP_x[e^{-q T_a^-};T_a^-<\infty]+h(\infty)\PP_x[e^{-q\zeta};\zeta<T_a^-],\quad a\leq x.$$ 
All the claims follow. 
	\end{proof}
In order to identify $\mathcal{I}^q$ (to the extent indicated in the Introduction) we will also need the following technical result. 
	\begin{lemma}\label{lemma:technical}
Assume $\mu r_{-1}>0$. The equation $q=\mu(\rr(z)-1)$ in $z\in (0,1)$ has a unique root $\phi_q$. The equation $0=\mu(\rr(z)-1)$ in $z\in (0,1]$ has at most two roots, one of which is $1$; the smaller of the two is denoted $\phi$. The map $q\mapsto \phi_q$ is a continuous and strictly decreasing bijection from $(0,\infty)$ onto $(0,\phi)$; $\phi=\uparrow\!\!\text{-}\lim_{q\downarrow 0}\phi_q$. 	We have that $q+\mu(1-\rr)$ is $>0$ (resp. $<0$) on $(\phi_q,1]$ (resp. $(0,\phi_q)$). Besides,  $\phi_q=\PP_1(T_0^-<\ee_q)$ if, ceteris paribus, we set $\lambda=0$.
\end{lemma}
\begin{proof}
	These are well-known facts from the theory of (homogeneous-Poisson-process--subordinated) left-continuous random walks, see e.g. \cite{vidmar-avram,vidmar2013fluctuation}. 
\end{proof}

\begin{definition}
	When $\mu r_{-1}=0$, we set $\phi_q:=\phi:=0$.  Otherwise $\phi_q$ and $\phi$ are given by the preceding lemma. It will sometimes be convenient to write $\phi_0:=\phi$. 
\end{definition}

We turn now to the identification of $\mathcal{I}^q$ (when $\phi_q<\varphi$). Before stating, and proving the result, let us provide some motivational computations, which also represent the gist of the eventual proof. 

Take any $x\in \mathbb{N}_0$. Recall that when $\mu=0$, then $\PP_x(T_0^-<\infty)=\varphi^x$, while if, ceteris paribus, $\lambda=0$, then $\PP_x(T_0^-<\ee_q)=\phi_q^x$. Given also the structure of \eqref{eq:harmonic}, it seems therefore natural, in order to find a harmonic function $f$ for $X^q$, to take the ansatz $$f(x)=\int_\alpha^\beta w(v)v^x\dd v$$ for suitable delimiters $0\leq \alpha<\beta\leq 1$ and a sufficiently nice $w:(\alpha,\beta)\to \mathbb{R}$. Plugging this in \eqref{eq:harmonic} one obtains (we omit technical reservations)
$$\int_\alpha^\beta[ v^x (q+\mu)w(x)+xv^{x-1}\lambda v w(v)]\dd v=\int_\alpha^\beta [ xv^{x-1}\lambda \pp(v)w(v)+v^x\mu\rr(v)]\dd v.$$
After an integration by parts and some rearranging it becomes 
$$\int_\alpha^\beta[ v^x (q+\mu(1-\rr(v)))w(x)]\dd v=\lambda(\pp(v)-v)w(v)v^x\vert_{v=\alpha}^\beta-\int_\alpha^\beta  v^x \frac{\dd}{\dd v}[\lambda(\pp(v)-v)w(v)]\dd v.$$
So, ignoring the boundary terms (hoping they vanish, or can otherwise be ``remedied'' by adding a simple function to $f$, like a constant or some power function), one should like to have $$(q+\mu(1-\rr(v)))w(v)=-\frac{\dd}{\dd v}[\lambda(\pp(v)-v)w(v)],\quad v\in (\alpha,\beta).$$
Basically this o.d.e. can be solved for $w$ and then the $\alpha$, $\beta$ can be conjured to provide the candidate solutions (modulo adding a constant to $f$). Backtracking one checks that they are indeed correct. However, at least to the best of the author's abilities, this programme can only be effected if $\phi_q<\varphi$ (the condition makes sure that all the technical bits which have been ignored in the preceding do in fact work out).  

Here is then our main result: the promised identification of $\mathcal{I}^q$ when $\phi_q<\varphi$. Recall that $q\in (0,\infty)$ throughout this section. 

\begin{theorem}\label{thm:skip-free}
Put 
$$\rho(v):=\lambda\vert \pp(v)-v\vert\text{ and }\gamma_q(v):=\frac{q+\mu(1-\rr(v))}{\rho(v)},\quad v\in (0,1)\backslash \{\varphi\},$$
noting that  $\pp(z)-z$ is $>0$ (resp. $<0$) for $z\in (0,\varphi)$ (resp. $z\in (\varphi,1)$), and that $\gamma_q(z)$ is $<0$ (resp. $>0$) for $z\in (0,\phi_q)\backslash \{\varphi\}$ (resp.  $z\in (\phi_q,1)\backslash \{\varphi\}$).

\begin{enumerate}[(i)]
	\item\label{thm:i} Assume  $\phi_q<\varphi$. Then  $\Phi_q:\mathbb{N}_0\to (0,\infty)$, given by
	\begin{equation*}\label{eq:main}
	\Phi_q(x):=q\int_0^\varphi \frac{\exp\{-\int_{\phi_q}^v\gamma_q(w)\dd w\}}{\rho(v)}v^x\dd v, \quad x\in \mathbb{N}_0,
	\end{equation*}
	is from $\II0\backslash\{0\}$. Assume further \eqref{eq:e}. Then $\Psi_q:\mathbb{N}_0\to \mathbb{R}$, given by
	\begin{equation*}
	\Psi_q(x):=1-q\int_\varphi^1 \frac{\exp\{-\int_{v}^1\gamma_q(w)\dd w\}}{\rho(v)}v^x\dd v, \quad x\in \mathbb{N}_0,
	\end{equation*} 	
	is an element of $\mathcal{I}^q\backslash \II0$. 
	\item\label{thm:skip-free:2} When $\phi_q=\varphi$, then $\Phi_q:\mathbb{N}_0\to (0,\infty)$, given by
	\begin{equation*}
	\Phi_q(x):=\varphi^x,\quad x\in \mathbb{N}_0,
	\end{equation*}
	is from $\II0\backslash\{0\}$.
	
\end{enumerate}
\end{theorem}
\begin{definition}\label{definition:phi-psi}
	In what follows the functions $\Phi_q$ and $\Psi_q$ are as introduced in Theorem~\ref{thm:skip-free}.
\end{definition}
When applied to $\Phi_q$ in lieu of $f$, Proposition~\ref{lemma:mtg} is a refinement of \cite[Lemma that includes Eq.~(2.7)]{Avram2019}. We provide some remarks and an example before turning to the proof. 
\begin{remark}\label{remark:fundamental-for-restriction}
The condition $\phi_q<\varphi$ is vacuous when $\mu r_{-1}=0$ or $\varphi=1$, and more generally whenever $\phi\leq \varphi$. If $\varphi<\phi$ (which necessitates supercritical branching and the presence of culling), then  $\phi_q<\varphi$ is equivalent to $q>\mu(\rr(\varphi)-1)$. 
	\end{remark}
\begin{remark}
Assume $\phi_q<\varphi$. Then $\int_{\phi_q}^\varphi\gamma_q=\infty$. Also,  $\int_0^{\phi_q}\gamma_q=-\infty$ unless $\mu r_{-1}=0$, in which case $\int_0^v\vert \gamma_q\vert <\infty$ for all $v\in (0,\varphi)$. If further \eqref{eq:e} prevails, then $\int_\varphi^1\gamma_q=\infty$ (but $\int_v^1\gamma_q<\infty$ for all $v\in (\varphi,1)$). 
	\end{remark}
\begin{remark}
	The delimiter $\phi_q$ in the integral for the expression of $\Phi_q$ is arbitrary to the extent that it may be replaced by any $\theta\in (0,\varphi)$ and $\Phi_q$ changes by a multiplicative constant only (but this does not mean that the condition $\phi_q<\varphi$ of \ref{thm:i} is superfluous). 
\end{remark}
\begin{remark}
Assume $\phi_q\in (0,\varphi)$. As, ceteris paribus, $\lambda\downarrow 0$, then, heuristically speaking, $\Phi_q$ is becoming ``more and more concentrated'' at $\mathbb{N}_0\ni x\mapsto \phi_q^x$, which is what one expects given the known results for the process $X$ for which, again ceteris paribus, $\lambda=0$.
\end{remark}
\begin{remark}\label{remark:simplify}
Assume $\mu=0$. Then an integration by parts simplifies the expressions for $\Phi_q$ and $\Psi_q$: for $x\in \mathbb{N}_0$,
	\begin{align*}
\Phi_q(x)&=\delta_{x0}+x\int_0^\varphi v^{x-1} \exp\left\{-\int_{0}^v\gamma_q\right\}\dd v, \text{ while, assuming \eqref{eq:e}},\\
	\Psi_q(x)&=x\int_\varphi^1v^{x-1}\exp\left\{-\int_{v}^1\gamma_q\right\}\dd v.
\end{align*}
We see that in this case $\Psi_q$  belongs even to $\IIinf\backslash\{0\}$ (under \eqref{eq:e}). 
\end{remark}
\begin{example}
Let $\gamma\in \mathbb{R}$, $\delta\in[0,\infty)$,  $\gamma+\delta\geq 0$, $\{\alpha,\beta\}\subset (0,\infty)$, and assume $\lambda(\pp(z)-z)=\alpha(1-z)+\beta(1-z)^2$, while $\mu(\rr(z)-1)z=\gamma(1-z)+\delta(1-z)^2$ for $z\in (0,1]$. It corresponds to non-trivial subcritical binary branching and to immigration/culling of at most one individual at a time. In this case $X$ is skip-free upwards as well as downwards. We obtain (the proportionality constant depends only on $q$ and $\alpha,\beta,\gamma,\delta$, not on $x\in \mathbb{N}_0$)
$$\Phi_q(x)\propto\frac{\Gamma \left(\frac{\gamma+\delta}{\alpha+\beta}+x+1\right)}{\Gamma \left(\frac{\gamma+\delta}{\alpha+\beta}+\frac{q}{\alpha}+x+1\right)} \, _2F_1\left(\frac{\beta\gamma-\alpha\delta}{\beta(\alpha+\beta)}+\frac{q}{\alpha}+1,\frac{\gamma+\delta}{\alpha+\beta}+x+1;\frac{\gamma+\delta}{\alpha+\beta}+\frac{q}{\alpha}+x+1;\frac{\beta}{\alpha+\beta}\right),$$
where $_2F_1$ is Gauss' hypergeometric function. Taking $\gamma=\delta=0$   one has consistency with \cite[Section~3.1]{Avram2019} (via Euler's transformation for $_2F_1$).
\end{example}
\begin{proof}[Proof of Theorem~\ref{thm:skip-free}]
\ref{thm:skip-free:2} is clear from Proposition~\ref{lemma:mtg}\ref{lemma:mtg:3}. Assume then $\phi_q<\varphi$. Set $\omega:=\exp\{-\int_{\phi_q}^\cdot\gamma_q\}$. Then $\omega:(0,\varphi)\to (0,1]$ is $C^1$,  it satisfies $\lim_{\varphi-}\omega=0$ and $\lim_{0+} \omega\in [0,\infty)$, and it solves the o.d.e. $-\gamma_q\omega=\omega'$. Besides, we see that $\Phi_q$ is well-defined (thanks to $\phi_q<\varphi$ in particular), bounded, non-zero, and vanishing at infinity. Using Proposition~\ref{lemma:mtg}\ref{lemma:mtg:3} one concludes via an integration by parts (to get rid of the ``$\lambda x$''s; see the computation immediately preceding the statement of Theorem~\ref{thm:skip-free}) that $\Phi_q\in \II0$. In a similar way one sees that $\Psi_q\in \mathcal{I}^q\backslash \II0$. 
\end{proof}

An  immediate consequence of Theorem~\ref{thm:skip-free} and Proposition~\ref{lemma:mtg} is the following martingale change of measure. We will have more to say on this in Proposition~\ref{proposition:measure-change}. 
\begin{corollary}\label{corollary:measure-change}
	Suppose  $\phi>\varphi$, i.e. $\rr(\varphi)>1$; then $\mu(\rr(\varphi)-1)\in (0,\infty)$, hence we may, and do take $q=\mu(\rr(\varphi)-1)$. Define the probabilities $$\QQ_x:=\left(\varphi^{-x}e^{-qT_0^-}\mathbbm{1}_{\{T_0^-<\infty\}}\right)\cdot \PP_x,\quad x\in \mathbb{N}_0.$$ Let $\FF$ be the natural filtration of $X$. Then: for any $\FF$-stopping time $S$ and any  $F\in (\FF_S/\mathcal{B}_{[0,\infty]})\cup ( b\FF_S)$, $$\QQ_x[F]=\PP_x[F\varphi^{X_S-x}e^{-q(S\land T_0^-)}],\quad x\in \mathbb{N}_0,$$
	where we set (a.s.) $X_S:=\lim_\infty X$ on $\{S=\infty\}$;  $X$ remains a minimal ctMc under the probabilities $\QQ=(\QQ_x)_{x\in \mathbb{N}_0}$. In fact, the system $(X,\QQ)$ is a ctBGWp with subcritical branching mechanism $(\mathbb{N}_0\ni k\mapsto p_k\varphi^{k-1})$, the same reproduction rate $\lambda$, immigration/culling mechanism $(\mathbb{N}\cup\{-1\}\ni k\mapsto r_k\varphi^k/\rr(\varphi))$, and immigration/culling rate $\mu\rr(\varphi)$.  Finally, $\QQ_x(T_0^-<\infty)=1$  for all $x\in \mathbb{N}_0$.
\end{corollary}
\begin{proof}
Note that $\phi_q=\varphi\in (0,1)$. By Proposition~\ref{lemma:mtg} and Theorem~\ref{thm:skip-free}\ref{thm:skip-free:2}, the process $W^q_{\Phi_q}$ of Proposition~\ref{lemma:mtg} is a bounded martingale with terminal value $e^{-qT_0^-}\mathbbm{1}_{\{T_0^-<\infty\}}$ and initial value $\varphi^x$, $\PP_x$-a.s. for all $x\in \mathbb{N}_0$. The observation that the $\QQ_x$, $x\in \mathbb{N}_0$, are probabilities and the first claim are then the content of a standard change of measure of the Markov process $X$ by the uniformly integrable martingale $W^q_{\Phi_q}$. The second claim follows by direct computation: for $t\in (0,\infty)$, $x\in \mathbb{N}$ and $y\in x+(\mathbb{N}_0\cup \{-1\})$, $\QQ_x(X_t=y)=\varphi^{-x}\PP_x[e^{-q(t\land T_0^-)}\varphi^{X_t};X_t=y]=\varphi^{y-x}\PP_x(X^q_t=y)$. The final claim follows directly from the definition of $\QQ_x$ and from Theorem~\ref{thm:skip-free}\ref{thm:skip-free:2}. 
\end{proof}

\section{Laplace transforms of first passage downwards and explosion times}\label{section:laplace}
The main harvest of Theorem~\ref{thm:skip-free} are the Laplace transforms of the extinction and explosion time of $X$. For their succinct formulation we complement Definition~\ref{definition:phi-psi} with

\begin{definition}\label{definition:phi0}
	When $\phi\leq \varphi$ set for $x\in \mathbb{N}_0$:
	$$\Phi_0(x):=\mu\int_0^\varphi\frac{\exp\left\{-\int_{\phi}^v\frac{\mu(1-\rr(w))}{\lambda(\pp(w)-w)}\dd w\right\}}{\lambda(\pp(v)-v)}v^x\dd v$$
	if $\mu(1-\rr(\varphi))>0$ [i.e. $\mu>0$ and $\phi<\varphi<1$]  or if $\mu>0$, $\phi<1=\varphi$ and $\int_0^1\frac{\exp\left\{-\int_{\phi}^v\frac{\mu(1-\rr(w))}{\lambda(\pp(w)-w)}\dd w\right\}}{\lambda(\pp(v)-v)}\dd v<\infty;$
	$\Phi_0(x):=\varphi^x$ otherwise [i.e. if $\mu=0$, or else $\phi=\varphi$, or else $\varphi=1$ and $\int_0^1\frac{\exp\left\{-\int_{\phi}^v\frac{\mu(1-\rr(w))}{\lambda(\pp(w)-w)}\dd w\right\}}{\lambda(\pp(v)-v)}\dd v=\infty$]. Note the conditions on the parameters ensure that $\Phi_0$ is always well-defined, strictly positive and finite.
\end{definition}

\begin{theorem}\label{corollary:laplace}
	Let $\{a,x\}\subset \mathbb{N}_0$, $a\leq x$, $q\in [0,\infty)$. If $q\geq \mu(\rr(\varphi)-1)$, i.e. if $\phi_q\leq\varphi$, then
	\begin{equation}\label{eq:first-passage}
	\PP_x[e^{-q T_a^-};T_a^-<\infty]=\frac{\Phi_q(x)}{\Phi_q(a)}.
	\end{equation}
	In particular if $\phi\leq \varphi$, then 
	\begin{equation}\label{eq:extinction}
	\PP_x(T_a^-<\infty)=\frac{\Phi_0(x)}{\Phi_0(a)},
	\end{equation}
	and  $\PP_x(T_0^-<\infty)=1$ for some (equivalently, all) $x\in \mathbb{N}$ iff 
	\begin{equation}\label{condition:for-extinction}
\varphi=1 \text{ and } \left(\mu=0\text{ or else }\phi=1\text{ or else }\int_0^1\frac{\exp\left\{-\int_{\phi}^v\frac{\mu(1-\rr(w))}{\lambda(\pp(w)-w)}\dd w\right\}}{\lambda(\pp(v)-v)}\dd v=\infty\right).
	\end{equation} 
	If \eqref{eq:e} prevails and $q\in (0,\infty)$, then for $q>\mu(\rr(\varphi)-1)$, i.e. for $\phi_q<\varphi$, 
	\begin{equation}\label{eq:explosion}
	\PP_x[e^{-q \zeta};\zeta<T_a^-]=\Psi_q(x)-\Psi_q(a)\frac{\Phi_q(x)}{\Phi_q(a)}.
	\end{equation}
	Finally, if $\phi\leq \varphi$ and \eqref{eq:e} holds true, then 
	\begin{equation}	\label{eq:explosion-pty}
	\PP_x(\zeta<T_a^-)+\PP_x(T_a^-<\infty)=1.
	\end{equation}
\end{theorem}
The integral condition of \eqref{condition:for-extinction} is similar to the recurrence criterion in the setting of cbi \cite[Theorem~3]{ma}. Before turning to the proof, let us make a number of remarks, and provide some examples. 
\begin{remark}\label{rmk:analytic}
	The theorem is  conclusive in case $\phi\leq \varphi$. Suppose, however, $\varphi<\phi$. The l.h.s.  of \eqref{eq:first-passage} is analytic in $\Re(q)>0$. Therefore, provided that one is able to recognize the r.h.s. of \eqref{eq:first-passage} as $f(q)$ for an analytic function $f$ defined for complex arguments, whose real part is $>0$, and having an explicit form --- which may or may not be easy/possible --- then for $q\in (0,\mu(\rr(\varphi)-1))$ an explicit expression for the l.h.s. of  \eqref{eq:first-passage} is got by analytic continuation, namely this expression is just $f(q)$. The l.h.s.  of \eqref{eq:first-passage} is also continuous in $q\in [0,\infty)$. Therefore, if the r.h.s. of \eqref{eq:first-passage} is known for $q>0$ and an explicit expression for the limit of this as $q\downarrow 0$ can be obtained, then we get an explicit value at $q=0$ of the l.h.s. An entirely analogous remark pertains to \eqref{eq:explosion}.  In any event (dropping now the assumption $\varphi<\phi$) Eqs.~\eqref{eq:first-passage}-\eqref{eq:explosion} determine the Laplace transforms of $T_a^-\vert_{\{T_a^-<\infty\}}$ and of $\zeta\vert_{\{\zeta<T_a^-\}}$ on a neighborhood of infinity, hence \cite[Theorem~8.4]{bhattacharya} uniquely the (subprobability) laws of these random variables. By a coupling argument (see the proof of Theorem~\ref{corollary:laplace} for the idea), we also have from \eqref{eq:extinction} that $\PP_x(T_a^-<\infty)\in [\varphi^{x-a},\phi^{x-a}]$ when $\varphi<\phi$ (and  $\PP_x(T_a^-<\infty)\in [\phi^{x-a},\varphi^{x-a}]$ when $\phi<\varphi$) for $a\leq x$ from $\mathbb{N}_0$.  
\end{remark}
\begin{example}
	As an illustration of Remark~\ref{rmk:analytic}, consider the binary branching mechanism with $p_0=1-p_2=\frac{1}{3}$, $\lambda=3$, and the immigration/culling mechanism with $r_{-1}=1$ and immigration/culling rate $\mu=1$. This is a case of a supercritical branching mechanism with culling and no immigration for which $1=\phi>\varphi=\frac{1}{2}$. Then one computes using \eqref{eq:first-passage} that $\PP_1[e^{-q T_0^-};T_0^-<\infty]=\frac{1}{4} \left(3-q+2^q (q-1) q \int_0^{1/2}\frac{v^q}{1-v}\dd v\right)$ for $q\in[1,\infty)$. The formula persists for $q\in (0,1)$ by analytic continuation, and then also at $q=0$ by continuity. In particular $\PP_1(T^-_0<\infty)=\frac{3}{4}$, which is indeed an element of $[\varphi,\phi]$.
\end{example}
\begin{example}
	Let $\pp(z)-z=(1-z)^2/2$, $z\in (0,1]$ (critical binary branching) and $1-\rr(z)=\sqrt{1-z}$,  $z\in (0,1]$ (immigration of infinite mean, no culling). We have $\phi=0<1=\varphi$ and $\int_0^1\frac{\exp\left\{-\int_{\phi}^v\frac{\mu(1-\rr(w))}{\lambda(\pp(w)-w)}\dd w\right\}}{\lambda(\pp(v)-v)}\dd v<\infty$, viz. \eqref{condition:for-extinction} fails. 
\end{example}
\begin{remark}
We will see in Corollary~\ref{corollary:critical} that there are examples of critical branching mechanisms with finite-in-the-mean immigration  for which extinction is not almost certain. Assume on the other hand that $\pp'(1-)<1$ and $0<\mu \rr'(1-)<\infty$ (i.e. we have, respectively, subcritical branching and on average the number of immigrants/culled individuals is strictly positive but finite). Then $\phi<1=\varphi$ and  $\int_0^1\frac{\exp\left\{-\int_{\phi}^v\frac{\mu(1-\rr(w))}{\lambda(\pp(w)-w)}\dd w\right\}}{\lambda(\pp(v)-v)}\dd v=\infty$, hence extinction is almost certain. Whether or not already merely subcritical branching is sufficient to ensure \eqref{condition:for-extinction} is not so obvious.
\end{remark}

\begin{remark}
What \eqref{eq:explosion-pty} is saying, is that for ctBGWp with immigration and culling satisfying $\phi\leq\varphi$ it cannot happen that with a positive probability $X$ explodes and with a positive probability it drifts to $\infty$ without exploding. A priori one cannot preclude the failure of \eqref{eq:explosion-pty} (it is easy to think of (even irreducible!) ctMc on $\mathbb{N}$ which display such behavior).
\end{remark}
\begin{remark}\label{remark:grey}
We wish to draw attention to some expressions involving $\Phi_q$, $\Psi_q$ and the semigroup of the process.	

\noindent Assume first $\mu=0$ and $q>0$. Let  $d\in \mathbb{N}_0$. One has that $$\PP_d(T_0^-\leq \ee_q)=q\int_0^\infty e^{-qr}\PP_d(X_r=0)\dd r= q\int_0^\infty e^{-qr}\PP_1(X_r=0)^d\dd r,$$ as noted in \cite[Eq.~(1.10)]{Avram2019}; the law of the extinction time may be implicitly characterized via  
	$\int_0^{\PP_1(T_0^-<t)}\frac{\dd s}{\lambda( \pp(s)-s)}=t$, $t\in [0,\infty)$ \cite[Theorem~1.2.3.2]{lambert}.   Similarly, $$\PP_d(\zeta> \ee_q)=q\int_0^\infty e^{-qr}\PP_d(X_r<\infty)\dd r=q\int_0^\infty e^{-qr}(1-\PP_1(X_r=\infty))^d\dd r.$$ In principle  one can therefore compute the quantities in the display (hence the Laplace transforms of $T_0^-$ and $\zeta$) provided the fixed-time extinction/explosion probabilities of $X$ are known in explicit form,  but this is rarely the case.  
	
\noindent Dropping now the assumption that $\mu=0$, the first equalities in the formulas in the preceding two displays certainly still hold true. Therefore we can flip the reasoning around: one can view $\Phi_q$ and $\Psi_q$ as furnishing,  via the above identifications, the Laplace transforms of, respectively, $([0,\infty)\ni t\mapsto \PP_d(X_t=0))$ and $([0,\infty)\ni t\mapsto \PP_d(X_t<\infty))$ (at all sufficiently large values of $q$). 
\end{remark}
\begin{remark}
Suppose $\mu=0$. Fix $a\in\mathbb{N}_0$ and $q\in (0,\infty)$. Through \eqref{eq:first-passage} the sequence $(\PP_x(T_a^-<\mathsf{e}_q))_{x\in \mathbb{N}_{>a}}$ determines $\Phi_q(x)/\Phi_q(a)$, $x\in \mathbb{N}_{>a}$. By the first equality of Remark~\ref{remark:simplify} these determine the moments of the  finite measure $\LL(\dd v):=\frac{1}{\Phi_q(a)}\mathbbm{1}_{(0,\varphi]}(v)v^{a}\exp\left\{-\int_{0}^v\gamma_q\right\}\mathrm{d} v$, hence they determine the measure $\LL$ itself  ($\because$ a  moment sequence determines a finite measure on $(0,1]$, which follows by functional monotone class; the Hausdorff moment  problem is always determinate). But (again by the first equality of Remark~\ref{remark:simplify}) $(a+1)\LL((0,\varphi])=\Phi_q(a+1)$; therefore the following are determined in turn: $\Phi_q(a+1)$,  $\Phi_q(a)$, $\mathbbm{1}_{(0,\varphi]}(v)v^{a}\exp\left\{-\int_{0}^v\gamma_q\right\}\mathrm{d} v$,  $\gamma_q$,  $\rho$, finally the law of $X$. In other words, if for two  ctBGWp their respective sequences $(\PP_x(T_a^-<\mathsf{e}_q))_{x\in \mathbb{N}_{>a}}$ agree --- which is relatively scarce information concerning the $\PP_x$-, $x\in \mathbb{N}_{>a}$, -laws of $T_a^-$ --- then these two processes have the same law. It is no longer true if we allow immigration, as will be clear from Corollary~\ref{corollary:critical} (but this does not preclude that more information concerning the laws of the $T_a^-$ could not be enough to determine the law of $X$ also when immigration/culling is allowed -- we do not go down this avenue further here). 
\end{remark}
\begin{remark}\label{remark:resurrection}
As far as  the content of the preceding theorem is concerned, the assumption that the process is stopped on hitting zero is in some sense superfluous, indeed the expectations/probabilities include only quantities that are determined prior to, or at the time when the process hits zero. Though, if the process $X$ were allowed to resurrect from zero (thanks to immigration), then explosion could happen after extinction, which means that \eqref{eq:explosion} with $a=0$ would no longer give the Laplace transform (on a neighborhood of infinity) of the explosion  time, but rather (merely) the Laplace transform (on a neighborhood of infinity) of the explosion time on the event that it happen before (the first) extinction. That having been said, by the strong Markov property, the process $X$, which is  allowed --- ceteris paribus --- to resurrect from zero (no culling at zero, of course), may be identified as a concatenation of a sequence of independent copies of the base case when the process  is stopped on hitting zero, but with each of these copies (except the first one) started according to $r(\cdot\vert \mathbb{N})$ ($r$ conditioned on $\mathbb{N}$) and interspersed with independent exponentially distributed random times (lengths of time required to resurrect). In this way the ``resurrection allowed'' case is (in principle) reduced to the standing case.
\end{remark}
\begin{proof}[Proof of Theorem~\ref{corollary:laplace}]
	\eqref{eq:explosion}, and  \eqref{eq:first-passage} for $q>0$ are immediate from Theorem~\ref{thm:skip-free} and Proposition~\ref{proposition:uniqueness}. As concerns \eqref{eq:first-passage} when $q=0$, i.e. \eqref{eq:extinction}, we have as follows. The case $\mu=0$ is well-known; we assume $\mu>0$.

	$(\bullet)$ For  $\phi=\varphi<1$, one appeals directly to the obvious extensions of Propositions~\ref{lemma:mtg} and~\ref{proposition:uniqueness} to $q=0$: the map $(\mathbb{N}_0\ni x\mapsto \varphi^x)$ is bounded, positive, harmonic for $X$ and vanishes at infinity. Thus \eqref{eq:extinction} follows in this case.
	
	$(\bullet)$ For $\phi=\varphi=1$	one can use a limiting coupling argument. Adjust indeed, ceteris paribus, the probabilities $p$ to $p'$  by  decreasing $p_0$ to $p_0'$, rescaling the remaining $p$s, and $r$ to $r'$ by decreasing $r_{-1}$ and rescaling the remaining $r$s, in such a way that (in the obvious notation) $\phi'=\varphi'\uparrow\uparrow 1$ as the adjustment is becoming lesser and lesser. Then $\PP_x(T_a^-<\infty)\geq \PP_x'(T_a^-<\infty)=(\varphi')^{x-a}\uparrow 1$, again as the adjustment is becoming lesser and lesser, and  \eqref{eq:extinction} is proved also for $\phi=\varphi=1$.
	
	$(\bullet)$ 	In case $\phi<\varphi<1$, hence $\mu(1-\rr(\varphi))>0$, \eqref{eq:extinction} follows from \eqref{eq:first-passage} with $q>0$ by dominated convergence upon taking the limit $q\downarrow 0$. 
	
	%
	
	$(\bullet)$ When $\phi<1=\varphi$, taking the limit $q\downarrow 0$ in  \eqref{eq:first-passage} with $q>0$ is   more delicate. On the one hand, splitting the integral for $\Phi_q(x)/q$ from Theorem~\ref{thm:skip-free} into $\int_0^{\phi}$ and $\int_{\phi}^1$, one can take the limit $q\downarrow 0$ in the first, resp. second of these, using bounded, resp. monotone convergence, to find that, as $q\downarrow 0$,
	$$\Phi_q(x)/q\to\int_0^1\frac{\exp\left\{-\int_{\phi}^v\frac{\mu(1-\rr(w))}{\lambda(\pp(w)-w)}\dd w\right\}}{\lambda(\pp(v)-v)}v^x\dd v.$$ Thus if $\int_0^1\frac{\exp\left\{-\int_{\phi}^v\frac{\mu(1-\rr(w))}{\lambda(\pp(w)-w)}\dd w\right\}}{\lambda(\pp(v)-v)}\dd v<\infty$, \eqref{eq:extinction} follows. On the other hand, assume now $\int_0^1\frac{\exp\left\{-\int_{\phi}^v\frac{\mu(1-\rr(w))}{\lambda(\pp(w)-w)}\dd w\right\}}{\lambda(\pp(v)-v)}\dd v=\infty$ and let $\epsilon\in (0,1-\phi)$. Then
	\begin{align*}
	\PP_x[e^{-qT_a^-};T_a^-<\infty]&=\frac{\int_0^1\frac{\exp\left\{-\int_{\phi_q}^v\frac{q+\mu(1-\rr(w))}{\lambda(\pp(w)-w)}\dd w\right\}}{\lambda(\pp(v)-v)}v^x\dd v}{\int_0^1\frac{\exp\left\{-\int_{\phi_q}^v\frac{q+\mu(1-\rr(w))}{\lambda(\pp(w)-w)}\dd w\right\}}{\lambda(\pp(v)-v)}v^a\dd v}\\
	&\geq \frac{(1-\epsilon)^x}{1+\int_0^{1-\epsilon}\frac{\exp\left\{-\int_{\phi_q}^v\frac{q+\mu(1-\rr(w))}{\lambda(\pp(w)-w)}\dd w\right\}}{\lambda(\pp(v)-v)}v^a\dd v\Big/\int_{1-\epsilon}^1\frac{\exp\left\{-\int_{\phi_q}^v\frac{q+\mu(1-\rr(w))}{\lambda(\pp(w)-w)}\dd w\right\}}{\lambda(\pp(v)-v)}\dd v}\\
	&\to (1-\epsilon)^x\text{ as $q\downarrow 0$}\\
		&\hspace{1cm}\text{(by bounded/monotone convergence for the numerator/denominator)}\\
	&\uparrow 1\text{ as }\epsilon\downarrow 0.
	\end{align*}
This concludes the proof \eqref{eq:extinction} for all cases that can occur (all the time assuming $\phi\leq\varphi$ and $\mu>0$, of course). \eqref{condition:for-extinction} is immediate from  \eqref{eq:extinction}.	Finally, let us prove \eqref{eq:explosion-pty}.


	$(\bullet)$ When $\phi<\varphi$ it  follows at once by taking the limit $q\downarrow 0$ in \eqref{eq:explosion} (using monotone convergence for the integral in the expression for $\Psi_q$ of Theorem~\ref{thm:skip-free}). 
	
	$(\bullet)$ When $\phi=\varphi$ the taking of the limit $q\downarrow 0$ in \eqref{eq:explosion} is again a little more delicate. On the one hand, still,  $\frac{\Phi_q(x)}{\Phi_q(a)}=\PP_x(T_a^-<\ee_q)\to \PP_x(T_a^-<\infty) = \varphi^{x-a}$ as $q\downarrow 0$. On the other hand, set $I(v):=\int_v^1\frac{\dd w}{\lambda(w-\pp(w))}$ for $v\in (\varphi,1)$; then $I:(\varphi,1)\to(0,\infty)$ is a strictly decreasing $C^1$ bijection and by a change of variables
	\begin{align*}
	1-\Psi_q(x)&=\int_\varphi^1\frac{q}{\lambda(v-\pp(v))}\exp\left\{-\int_v^1\frac{q+\mu(1-\rr(w))}{\lambda(w-\pp(w))}\dd w\right\}v^x\dd v\\
	&=\int_0^\infty\dd z q e^{-qz}\exp\left\{-\int_{I^{-1}(z)}^1\frac{\mu(1-\rr(w))}{\lambda(w-\pp(w))}\dd w\right\}I^{-1}(z)^x\\
	&=\PP_0\left[\exp\left\{-\int_{I^{-1}(\ee_q)}^1\frac{\mu(1-\rr(w))}{\lambda(w-\pp(w))}\dd w\right\}I^{-1}(\ee_q)^x\right]\\
	&\to \exp\left\{-\int_{\varphi}^1\frac{\mu(1-\rr(w))}{\lambda(w-\pp(w))}\dd w\right\}\varphi^x\text{ as }q\downarrow 0,
	\end{align*}
	by bounded convergence.
	Consequently we find that \eqref{eq:explosion-pty} holds true also in this case.
\end{proof}
Various consequences of Theorem~\ref{corollary:laplace} will be explored in the next section. Here we would like to note a strengthening of Corollary~\ref{corollary:measure-change} that allows immediately to extend \eqref{eq:first-passage} to include also the cumulative-lifetime-to-date (avalanche size, total progeny) process $\int_0^\cdot X_t\dd t$, Corollary~\ref{corollary:avalanche-size}. A preliminary lemma is required to this end; it is Lemma~\ref{lemma:technical} rewritten for $p_{\cdot+1}\vert_{\{-1\}\cup\mathbb{N}}$ in lieu of of $r$.  

\begin{lemma}
Let $\qbar\in (0,\infty)$. The equation $\frac{\lambda+\qbar}{\lambda}=\pp(z)/z$ in $z\in (0,1)$ has a unique root $\varphi_\qbar$. The map $\qbar\mapsto \varphi_\qbar$ is a continuous and strictly decreasing bijection from $(0,\infty)$ onto $(0,\varphi)$; $\varphi=\uparrow\!\!\text{-}\lim_{\qbar\downarrow 0}\varphi_\qbar$.\qed
\end{lemma}

\begin{definition}
In what follows $\varphi_\qbar$ is as given by the preceding lemma. On occasion we write $\varphi_0:=\varphi$. 
\end{definition}

\begin{proposition}\label{proposition:measure-change}
Let $\qbar\in [0,\infty)$ and set $q:=\mu(\rr(\varphi_\qbar)-1)\in[-\mu,\infty)$. Then the process $(\varphi_\qbar^{X_t}e^{-q(t\land T_0^-)-\qbar \int_0^{t}X_s\dd s})_{t\in [0,\infty)}$ is a martingale in the natural filtration $\FF=(\FF_t)_{t\in [0,\infty)}$ of $X$. Assuming the  filtered space $(\Omega,\FF)$ is rich enough for Kolmogorov's extension theorem to apply, see \cite[Section~V.4]{parthasarathy}\footnote{If one takes for $\Omega$ a suitable canonical space, then for sure the conditions of the extension theorem are met. This one can do without affecting any distributional results.}, for each $x\in \mathbb{N}_0$, there exists a unique probability measure $\QQ_x$ on $(\Omega,\FF_\infty)$ such that 
$$\QQ_x\vert_{\FF_t}=\left(\varphi_\qbar^{X_t-x}e^{-q(t\land T_0^-)-\qbar \int_0^{t}X_s\dd s}\right)\cdot \PP_x,\quad t\in [0,\infty).$$ Under the measures $\QQ=(\QQ_x)_{x\in \mathbb{N}_0}$ the process $X$ remains a minimal ctMc. In fact, the system $(X,\QQ)$ is again a ctBGWp with not-supercritical branching mechanism $(\mathbb{N}_0\ni k\mapsto p_k\varphi_\qbar^{k}/\pp(\varphi_\qbar))$,  reproduction rate $\lambda+\qbar$, immigration/culling mechanism  $(\mathbb{N}\cup\{-1\}\ni k\mapsto r_k\varphi_\qbar^k/\rr(\varphi_\qbar))$, and immigration/culling rate $q+\mu=\mu\rr(\varphi_\qbar)$. 
\end{proposition}
\begin{proof}
One checks that $(x(\lambda+\qbar)+\mu+q)\varphi_\qbar^x=\mu \sum_{k=-1}^\infty r_k\varphi_\qbar^{x+k}+\lambda x\sum_{k=0}^\infty p_k\varphi_\qbar^{x+k-1}$ for all $x\in \mathbb{N}$. The martingale property is then got as in Proposition~\ref{lemma:mtg}. All the claims follow.   
\end{proof}
\begin{corollary}\label{corollary:avalanche-size}
Let $\{a,x\}\subset \mathbb{N}_0$, $a\leq x$, $\{q,\qbar\}\subset [0,\infty)$. If $q>0$ or $\varphi_\qbar<1$, and if $q>\mu(\rr(\varphi_\qbar)-1)$ (the latter is equivalent to  $\phi_q<\varphi_\qbar$, when $\mu>0$), then

\begin{equation}\label{eq:with-avalanche-size}
\PP_x[e^{-qT_a^--\qbar\int_0^{T_a^-}X_s\dd s};T_a^-<\infty]=\frac{\Phi_{q,\qbar}(x)}{\Phi_{q,\qbar}(a)},
\end{equation}
where
$$\Phi_{q,\qbar}(z):=(q-\mu(\rr(\varphi_\qbar)-1))\int_0^{\varphi_\qbar} \frac{\exp\left\{-\int_{\phi_q}^v\frac{q+\mu(1-\rr(w))}{ \lambda(\pp(w)-w)-\qbar w}\dd w\right\}}{\lambda(\pp(v)-v)-\qbar v}v^z\dd v,\quad z\in \mathbb{N}_0.$$
Consequently, if $r_{-1}=0$, $\mu>0$ and $\qbar>0$,
\begin{equation*}
\PP_x[e^{-\qbar\int_0^{T_a^-}X_s\dd s};T_a^-<\infty]=\frac{\int_0^{\varphi_\qbar} \frac{\exp\left\{-\int_{0}^v\frac{\mu(1-\rr(w))}{ \lambda(\pp(w)-w)-\qbar w}\dd w\right\}}{\lambda(\pp(v)-v)-\qbar v}v^x\dd v}{\int_0^{\varphi_\qbar} \frac{\exp\left\{-\int_{0}^v\frac{\mu(1-\rr(w))}{ \lambda(\pp(w)-w)-\qbar w}\dd w\right\}}{\lambda(\pp(v)-v)-\qbar v}v^a\dd v},
\end{equation*}
while if $\mu=0$, then, with no further restriction on $\qbar$,
\begin{equation*}
\PP_x[e^{-\qbar\int_0^{T_a^-}X_s\dd s};T_a^-<\infty]=\varphi_\qbar^{x-a}.
\end{equation*}
	\end{corollary}
\begin{remark}
	The expression for $\PP_x[e^{-\qbar\int_0^{T_a^-}X_s\dd s};T_a^-<\infty]$ in case $\mu=0$ is also immediate from the Lamperti transform, cf. Remark~\ref{remark:lamperti}.
\end{remark}
\begin{remark}\label{rmk:akin}
We have already mentioned in the Introduction that \eqref{eq:with-avalanche-size} dovetails nicely with \cite[Eq.~(11)]{ma} from the setting of cbi. In order to see this in a little more detail fix  $\{h,\gamma\}\subset (0,\infty)$ and define $\hat{X}:=hX_{\gamma \cdot}$. On an heuristic level, for ``small'' $h$ and ``large'' $\gamma$, $\hat{X}$ may be thought of as an approximation of a csbp with immigration and culling. Next, set, for $z\in [0,\infty)$, $\psi(z):=\frac{\gamma\lambda}{h}(\pp(e^{-zh})/e^{-zh}-1)$ and $\Phi(z):=\mu\gamma(1-\rr(e^{-zh}))$, cf. \cite[Eqs. (3.60-3.61)]{li2010measure}. Further, let $\{\alpha,\overline{\alpha}\}\subset [0,\infty)$, with $\alpha\lor \overline{\alpha}>0$ if $\varphi_ 0=1$, and put $\pmb{q}(\overline{\alpha}):=-\log(\varphi_{\overline{\alpha} h/\gamma})/h$, $\pmb{\theta}(\alpha):=-\log(\phi_{\alpha/\gamma})/h$ [$\log 0:=-\infty$]. Then, provided $\pmb{q}(\overline{\alpha})<\pmb{\theta}(\alpha)$, for $\{a,x\}\subset h\mathbb{N}_0$ with $a\leq x$,  in the obvious notation, $\hat{\PP}_x[e^{-\alpha \hat{T}_a^--\overline{\alpha}\int_0^{\hat{T}_a^-}\hat{X}_t\dd t};\hat{T}_a^-<\infty]=\PP_{x/h}[e^{-\frac{\alpha}{\gamma}  T_{a/h}^--\frac{\overline{\alpha} h}{\gamma}\int_0^{T_{a/h}^-}X_t\dd t};T_{a/h}^-<\infty]=\frac{\hat{\Phi}_{\alpha,\overline{\alpha}}(x)}{\hat{\Phi}_{\alpha,\overline{\alpha}}(a)}$, where $$\hat{\Phi}_{\alpha,\overline{\alpha}}(y):=\int_{\pmb{q}(\overline{\alpha})}^\infty\frac{\exp\left\{\int_{\pmb{\theta}(\alpha)}^z\frac{\Phi(u)+\alpha}{\psi(u)-\overline{\alpha}}\dd u\right\}}{\psi(z)-\overline{\alpha}}e^{-yz}\dd z,\quad y\in h\mathbb{N}_0.$$ 
\end{remark}
\begin{proof}[Proof of Corollary~\ref{corollary:avalanche-size}]
	By optional sampling, using Proposition~\ref{proposition:measure-change} (with the same $\qbar$), $$\PP_x\left[e^{-q(T_a^-\land t)-\qbar \int_0^{T_a^-\land t}X_s\dd s}\varphi_\qbar^{X_{T_a^-\land t}-x}\right]=\QQ_x[e^{-(q-\mu(\rr(\varphi_\qbar)-1))(T_a^-\land t)}],\quad t\in [0,\infty).$$ Then let $t\to\infty$ and use \eqref{eq:first-passage} to get \eqref{eq:with-avalanche-size}. The second claim for $\mu>0$ follows by setting $q=0$; for $\mu=0$ it is obtained by taking the limit $q\downarrow 0$ after one has effected a change of variables of the kind that we have seen in the proof of Theorem~\ref{corollary:laplace} (we leave the details to the reader). 
\end{proof}

\section{Further consequences}\label{applications}
As a first offspring of Theorem~\ref{corollary:laplace} we have the following identification of a situation in which the overall infimum of $X$ before an independent exponential time is uniformly distributed.

\begin{corollary}\label{corollary:critical}
Let $q\in [0,\infty)$.  Suppose $\mu>0$, $\frac{\mu}{q+\mu}=\sum_{k=1}^\infty kp_k$, $\lambda=q+\mu$, $r_{-1}=0$ and $r_k=\frac{q+\mu}{\mu}(k+1)p_{k+1}$ for $k\in \mathbb{N}$. Then, for all $x\in \mathbb{N}_0$, $\inf_{t\in [0,\ee_q)} X_t$ is	uniformly distributed on $\{0,\ldots,x\}$ under $\PP_x$ and, for all $a\in\mathbb{N}_0$ with $a\leq x$,  $$\PP_x[e^{-\qbar\int_0^{T_a^-}(1+X_s)\dd s};T_a^-<\ee_q]=\frac{a+1}{x+1}\varphi_\qbar^{x-a},\quad \qbar\in [0,\infty).$$
\end{corollary} 
The analogue of the preceding for the case $q=0$ in the setting of cbi can be found in \cite[Corollaries~11 and~13]{ma}.
\begin{proof}
One checks readily that $\varphi=1$, $\phi=0$, and  $q+\mu(1-\rr)=-\lambda(\pp-\mathrm{id}_{(0,1)})'$ on $(0,1)$. Consequently $\Phi_q(x)/\Phi_q(0)=\frac{1}{x+1}$ for all $x\in \mathbb{N}_0$, whence the  first claim follows  from Theorem~\ref{corollary:laplace}. The second claim is got similarly from Corollary~\ref{corollary:avalanche-size}.
 \end{proof}

Differentiating the Laplace transform \eqref{eq:first-passage} we can obtain the means of the first passage times downwards (when they are finite a.s.). To ease the computations we make a simplifying assumption. 
\begin{corollary}\label{corollary:first-passage-mean}
	Assume the extinction time is a.s. finite and $\phi<1$. 	
	Let $\{x,a\}\subset \mathbb{N}_0$, $a< x$. Then 
	$$\PP_x[T_a^-]=\int_0^1(v^a-v^x)\frac{\exp\left\{\int_{v}^1\frac{\mu(1-\rr(w))}{\lambda(\pp(w)-w)}\dd w\right\}}{\lambda(\pp(v)-v)}\dd v.$$
\end{corollary}
For the cbi version, see  \cite[Corollary~9]{ma}.
\begin{proof}
Necessarily $\varphi=1$. By monotone convergence $\PP_x[T_a^-]=\lim_{q\downarrow 0}\frac{1-\PP_x[e^{-qT_a^-}]}{q}=\lim_{q\downarrow 0}\frac{\Phi_q(a)-\Phi_q(x)}{\Phi_q(a)q}$. Now, setting $I(v):=\int_{0}^v\frac{\dd w}{\lambda(\pp(w)-w)}$ for $v\in (0,1)$, so that $I:(0,1)\to (0,\infty)$ is a strictly increasing $C^1$ bijection,
	\begin{align*}
	\Phi_q(a)&=q\int_0^1\frac{\exp\left\{-\int_{\phi_q}^v\frac{q+\mu(1-\rr(w))}{\lambda(\pp(w)-w)}\dd w\right\}}{\lambda(\pp(v)-v)}v^a\dd v\\
	&=\int_0^\infty \dd zqe^{-q(z-I(\phi_q))}\exp\left\{-\int_{\phi_q}^{I^{-1}(z)}\frac{\mu(1-\rr(w))}{\lambda(\pp(w)-w)}\dd w\right\}I^{-1}(z)^a\\
	&=e^{qI(\phi_q)}\PP_a\left[\exp\left\{-\int_{\phi_q}^{I^{-1}(\ee_q)}\frac{\mu(1-\rr(w))}{\lambda(\pp(w)-w)}\dd w\right\}I^{-1}(\ee_q)^a\right]\\
	&\to \exp\left\{-\int_{\phi}^{1}\frac{\mu(1-\rr(w))}{\lambda(\pp(w)-w)}\dd w\right\}\text{ as }q\downarrow 0,
	\end{align*}
	by bounded convergence and since $qI(\phi_q)\to 0 \cdot I(\phi)=0$ as $q\downarrow 0$ thanks to $\phi<1$.
Besides, 
\begin{align*}
(\Phi_q(a)-\Phi_q(x))/q&=\int_0^1(v^a-v^x)\frac{\exp\left\{-\int_{\phi_q}^v\frac{q+\mu(1-\rr(w))}{\lambda(\pp(w)-w)}\dd w\right\}}{\lambda(\pp(v)-v)}\dd v\\	&\to \int_0^1(v^a-v^x)\frac{\exp\left\{-\int_{\phi}^v\frac{\mu(1-\rr(w))}{\lambda(\pp(w)-w)}\dd w\right\}}{\lambda(\pp(v)-v)}\dd v\text{ as }q\downarrow 0,
\end{align*}
by monotone convergence on $(\phi,1)$ and bounded convergence on $(0,\phi)$.
\end{proof}
More generally, at least whenever $\phi\leq \varphi$, similar computations could be effected to obtain the means of the explosion time (before passage downwards) and of the first passage times downwards, \emph{on the respective events on which these are finite}, also of the minima $\zeta\land T_a^-$, $a\in \mathbb{N}_0$. We give  a flavor of this in the next corollary. 

\begin{corollary}\label{corollary:explosion}
Let $\{x,a\}\subset \mathbb{N}_0$. Assume \eqref{eq:e} holds. If $\mu>0$ and $\phi<\varphi$ then $$
\PP_x[\zeta\land T_a^-]=\int_\varphi^1\left(v^x-v^a\frac{\int_0^\varphi\frac{\exp\left\{-\int_{\phi}^z\frac{\mu(1-\rr(w))}{\lambda(\pp(w)-w)}\dd w\right\}}{\lambda(\pp(z)-z)}z^x\dd z}{\int_0^\varphi\frac{\exp\left\{-\int_{\phi}^z\frac{\mu(1-\rr(w))}{\lambda(\pp(w)-w)}\dd w\right\}}{\lambda(\pp(z)-z)}z^a\dd z}\right)\frac{\exp\left(-\int_v^1\frac{\mu(1-\tilde{r}(w))}{\lambda(w-\tilde{p}(w))}\dd w\right)}{\lambda(v-\pp(v))}\dd v<\infty,\quad a\leq x.$$
If $\mu=0$ then (we restrict to $a=0$ for simplicity) $\PP_x[\zeta\land T_0^-]=\int_0^1\frac{v^x-\varphi^x}{\lambda(v-\pp(v))}\dd v<\infty$, in fact $\PP_x[\zeta;\zeta<\infty]=\int_\varphi^1\frac{v^x-\varphi^x}{\lambda(v-\pp(v))}\dd v<\infty$  and $\PP_x[T_0^-;T_0^-<\infty]=\int_0^\varphi\frac{\varphi^x-v^x}{\lambda(\pp(v)-v)}\dd v<\infty$. 
\end{corollary} 
Note the finiteness of the expectations, which is not obvious a priori. $\zeta\land T_0^-$ is (a.s.) nothing but the time when the process $X$ ceases to evolve (terminates at its final value, be it $0$ or $\infty$). One may compare the  expressions corresponding to $\mu=0$ with \cite[Lemma~5.2]{li-infinite} which has  infinite immigration and infinite offspring (both with positive probability).
\begin{proof}
Necessarily $\varphi<1$. Assume first $\mu>0$ and $\phi<\varphi$. By monotone convergence $\PP_x[\zeta;\zeta<T_a^-]=\lim_{q\downarrow 0}\frac{\PP_x(\zeta<T_a^-)-\PP_x[e^{-q\zeta};\zeta<T_a^-]}{q}$ and $\PP_x[T_a^-;T_a^-<\infty]= \lim_{q\downarrow 0}\frac{\PP_x(T_a^-<\infty)-\PP_x[e^{-qT_a^-};T_a^-<\infty]}{q}$. Then according to Theorems~\ref{corollary:laplace} and~\ref{thm:skip-free} we obtain
	\begin{align*}
\PP_x[\zeta\land T_a^-]&=\PP_x[\zeta;\zeta<T_a^-]+\PP_x[T_a^-;T_a^-<\infty]= \lim_{q\downarrow 0}\frac{1-\frac{\Phi_0(x)}{\Phi_0(a)}-\Psi_q(x)+\Psi_q(a)\frac{\Phi_q(x)}{\Phi_q(a)}+\frac{\Phi_0(x)}{\Phi_0(a)}-\frac{\Phi_q(x)}{\Phi_q(a)}}{q}\\
&=\lim_{q\downarrow 0}\frac{1-\left(1-q\int_\varphi^1\frac{\exp(-\int_v^1\gamma_q)}{\rho(v)}v^x\dd v\right)+\left(-q\int_\varphi^1\frac{\exp(-\int_v^1\gamma_q)}{\rho(v)}v^a\dd v\right)\frac{\Phi_q(x)}{\Phi_q(a)}}{q}\\
&=\lim_{q\downarrow 0}\int_\varphi^1\left(v^x-v^a\frac{\Phi_q(x)}{\Phi_q(a)}\right)\frac{\exp\left(-\int_v^1\gamma_q\right)}{\rho(v)}\dd v\\
&=\int_\varphi^1\left(v^x-v^a\frac{\Phi_0(x)}{\Phi_0(a)}\right)\frac{\exp\left(-\int_v^1\frac{\mu(1-\tilde{r}(w))}{\lambda(w-\tilde{p}(w))}\dd w\right)}{\lambda(v-\pp(v))}\dd v<\infty,
	\end{align*}
by dominated convergence, using \eqref{eq:ee} to handle the integral at $1-$ (the expression in the exponent is $\leq 0$ due to $\phi\leq \varphi$) and using $\mu(1-\tilde{r}(\varphi))>0$ ($\because$ $\mu>0$ and $\phi<\varphi$) to handle the integral at $\varphi+$, while away from $\varphi$ and $1$ one can use just bounded convergence. It remains to plug in the expression for $\Phi_0$ of Definition~\ref{definition:phi0}.

Assume now $\mu=0$. Still by monotone convergence $\PP_x[\zeta;\zeta<\infty]=\lim_{q\downarrow 0}\frac{\PP_x(\zeta<\infty)-\PP_x[e^{-q\zeta};\zeta<\infty]}{q}$. Using Theorem~\ref{corollary:laplace} and Remark~\ref{remark:simplify} we obtain
	\begin{align*}
\PP_x[\zeta;\zeta<\infty]&= \lim_{q\downarrow 0}\frac{1-\varphi^x-x\int_\varphi^1 v^{x-1}e^{-q\int_v^1\frac{\dd w}{\lambda(w-\pp(w))}}\dd v}{q}\\
&=\lim_{q\downarrow 0}\frac{x\int_\varphi^1 v^{x-1}\left(1-e^{-q\int_v^1\frac{\dd w}{\lambda(w-\pp(w))}}\right)\dd v}{q}\\
&= x\int_\varphi^1 v^{x-1}\int_v^1\frac{\dd w}{\lambda(w-\pp(w))}\dd v,\text{ by monotone convergence}. 
	\end{align*}
 The final expression for $\PP_x[\zeta;\zeta<\infty]$ follows upon an integration by parts and its finiteness from condition \eqref{eq:ee}. Likewise we get from $\PP_x[T_0^-;T_0^-<\infty]= \lim_{q\downarrow 0}\frac{\PP_x(T_0^-<\infty)-\PP_x[e^{-qT_0^-};T_0^-<\infty]}{q}$, using again  Theorem~\ref{corollary:laplace}, Remark~\ref{remark:simplify}, monotone convergence and integration by parts, that 
 	\begin{align*}
\PP_x[T_0^-;T_0^-<\infty]&= \lim_{q\downarrow 0}\frac{\varphi^x-\delta_{x0}-x\int_0^\varphi v^{x-1}e^{-q\int_0^v\frac{\dd w}{\lambda(\pp(w)-w)}}\dd v}{q}\\
&=\lim_{q\downarrow 0}\frac{x\int_0^\varphi v^{x-1}\left(1-e^{-q\int_0^v\frac{\dd w}{\lambda(\pp(w)-w)}}\right)\dd v}{q}\\
&= x\int_0^\varphi v^{x-1}\int_0^v\frac{\dd w}{\lambda(\pp(w)-w)}\dd v=\int_0^\varphi\frac{\varphi^x-v^x}{\lambda(\pp(v)-v)}\dd v<\infty,
	\end{align*}
	the finiteness being immediate (because $\varphi<1$ and due to the nature of the expression in the numerator).
\end{proof}

In addition to obtaining expectations, the explicit expressions for the Laplace transforms of Theorem~\ref{corollary:laplace} might  be analysed in terms of their asymptotics as $q\to \infty$ as well as for $q\downarrow 0$ (the latter when $\phi\leq \varphi$), which one would expect to yield asymptotics of the laws of the first passage and explosion times at $0+$ and  $\infty$, respectively, using the results of regular variation \cite{bgt}. We do not attempt this further analytical exercise here, leaving such computations to the interested party.\label{para:laplace-asmyp} 

Instead, we specify next the probabilities of $X$ ``conditioned to become extinct before an independent exponential random clock has rung''. 

\begin{corollary}\label{corollary:condition}
 Let $q\in[\mu(\rr(\varphi)-1)\lor 0,\infty)$. Let $(\Theta,\AA,\BB=(\BB_t)_{t\in [0,\infty)})$ be the canonical filtered probability space of $\mathbb{N}$-valued paths with lifetime, cemetery $0$. The canonical process on this space is denoted $Z=(Z_t)_{t\in [0,\infty)}$ and the lifetime $\xi$. Then there exists a unique family of probability measures $(\PP_x^\downarrow)_{x\in \mathbb{N}}$ on $(\Theta,\AA)$ such that for all $x\in \mathbb{N}$,  $t\in [0,\infty)$, and then $F\in \BB_t/\mathcal{B}_{[0,\infty]}$, $$\PP_x^\downarrow[F;t<\xi]:=\frac{e^{-qt}}{\Phi_q(x)}\PP_x[F(X)\Phi_q(X_t);t<T_0^-]=\PP_x[F(X)\mathbbm{1}_{\{t<T_0^-\}}\vert T_0^-<\ee_q].$$
	Under this family of measures, $Z$ is a (non-conservative) ctMc in the filtration $\BB$, lifetime $\xi$, cemetery $0$, state space $\mathbb{N}$, whose transition law is specified as follows: for $x\in \mathbb{N}$, the rate at which $Z$ leaves state $x$ is $q+\mu+\lambda x$; on leaving state $x$, the probability that it jumps to $x-1$ is $\frac{(p_0\lambda x+r_{-1}\mu)\Phi_q(x-1)}{(q+\mu+\lambda x)\Phi_q(x)}\mathbbm{1}_{[2,\infty)}(x)$, while for $k\in \mathbb{N}$, the probability that it jumps to $x+k$  is $\frac{(p_{k+1}\lambda x+\mu r_k)\Phi_q(x+k)}{(q+\mu+\lambda x)\Phi_q(x)}$, no other jumps (except to the cemetery) being possible with a positive probability; in particular $Z$ is sent to the cemetery only when in state $1$, and then at rate $(p_0\lambda+r_{-1}\mu)\frac{\Phi_q(0)}{\Phi_q(1)}$.
	\qed
\end{corollary}
\begin{proof}
	This is a standard Doob transform by an excessive function, see e.g. \cite[Chapter~11]{chung2006markov} (the condition $q\geq \mu(\rr(\varphi)-1)$ merely ensures that $\phi_q \leq \varphi$); apart from the eventual explicit expressions that one obtains, the particular stochastic dynamics of $X$ does not intervene beyond it being a (sufficiently nice) Markov process. The description of the stochastic dynamics of $Z$ follows by looking at $\lim_{t\downarrow 0}\frac{1-\PP^\downarrow_x(Z_t=x)}{t}$ and $\lim_{t\downarrow 0}\frac{\PP^\downarrow_x(Z_t=y)}{t}$ for the relevant $x$ and $y$.
\end{proof}
Finally we fix a $q\in [0,\infty)$ and an $x\in \mathbb{N}_0$, and consider  the temporal factorization at the minimum for $X$ on $[0,\ee_q)$, when issued from $x$.  We will couch the arguments  in the milieu of the specific dynamics of $X$ to take full advantage of the eventual (technical) simplifications --  but the  factorization of a real Markov process (conditionally) at the minimum is a much more general (and then technically demanding) theme, see e.g. \cite{millar}.

The following observation represents the basis of the argument:
\begin{quote}
Under $\PP_x$, the sequence of the consecutive excursions from strict new minima, $$\epsilon^k:=(X_{T_k^-+t})_{t\in [T_k^-,T_{k-1}^-)},\quad k =x,\ldots,0\text{ (note the reverse order!)},$$ where we take $T_{-1}^-:=\infty$,   has the law of a finite sequence $S=(S_l)_{l=x}^0$ of independent (but not identically distributed, unless $x=0$) path segments with lifetime, \emph{absorbed} into $\emptyset$ after first (in the order indicated above) entering a path segment of infinite length, and with the distribution of $S_l$ being $((X_{t})_{t\in [0,T_{l-1}^-)})_\star \PP_l$ for $l\in \{0,\ldots, x\}$. 
\end{quote}
For $t\in (0,\infty]$ set indeed 
$$G_t:=\sup\{u\in [0,t): X\text{ attains a strict new minimum at time }u\}$$ (we consider $0$ as being a time at which a strict new minimum occurs) and note that $G_{\ee_q}<\infty$ a.s. even for $q=0$.  We then have 
 
\begin{proposition}\label{proposition:independnece}
	The following two objects are independent under $\PP_x$, conditionally on $X_{G_{\ee_q}}$, the value of the last strict new minimum before $\ee_q$: (i) $(X_t)_{t\in [0,G_{\ee_q})}$, the path of $X$ seen until it makes its last strict new minimum before $\ee_q$;  and (ii) $(\epsilon^{X_{G_{\ee_q}}},\ee_q-G_{\ee_q})$, the excursion from strict new minima straddling $\ee_q$, together with the time that has elapsed from the last strict new minimum  to $\ee_q$. 
\end{proposition}
\begin{proof}
The well-known trick to show this is that instead of simply waiting for $\ee_q$ to ring, we may mark each excursion with an independent exponential clock of rate $q$, and wait for the first one to ring before the excursion ends. We include a precise argument for the sake of completeness.

Formally let us take then, in addition to the sequence $(S_l)_{l=x}^0$  that we have already introduced (recall, this one is ``non-absorbed'', consisting of independent path segments), also an independent sequence $(e_l)_{l=x}^0$ of i.i.-exponentially with rate $q$-d. random variables.  Let $K$ be the first index $l$ from $x,\ldots,0$ (in the indicated order) for which the lifetime of $S_l$ is $\geq e_l$. Because the lifetime of $S_0$ is a.s. infinite such an index exists (a.s.). Then, by the memoryless property of the exponential distribution, as far as the joint law is concerned, we may see (i) and (ii), as being, respectively, (i') the obvious ``concatenation'' of $S_x,\ldots,S_{K-1}$, and (ii') $(S_K,e_K)$. 

The desired conclusion now follows from  an elementary claim concerning first entries into measurable sets of sequences with independent values, to feature presently (Lemma~\ref{lemma:elementary}).
\end{proof}

\begin{lemma}\label{lemma:elementary}
	Let $Z=(Z_k)_{k\in \mathbb{N}_0}$ be a sequence of independent random elements valued in some measurable space $(E,\mathcal{E})$.  Denote by $\LL_k$ the law of $Z_k$, $k\in \mathbb{N}_0$. Let also $A\in \mathcal{E}$ and $K:=\inf \{k\in \mathbb{N}_0:Z_k\in A\}$. Then, for each $k\in \mathbb{N}_0$ with $\PP(K=k)>0$, conditionally on $\{K=k\}$, $Z_0,\ldots,Z_{k-1}, Z_k$ are independent with respective laws $\LL_0(\cdot\vert E\backslash A),\ldots,\LL_{k-1}(\cdot\vert E\backslash A),\LL_k(\cdot\vert A)$. \qed
\end{lemma}
As a consequence of the preceding we can provide an explicit description of the temporal quantities at the minimum.

\begin{corollary}\label{corollary:at-min}
Under $\PP_x$, $G_{\ee_q}$ and $\ee_q-G_{ \ee_q}$ are conditionally independent given $X_{G_{\ee_q}}$. Assume now $q\geq \mu(\rr(\varphi)-1)$, i.e. $\phi_q\leq\varphi$, and  let $k\in \{0,\ldots,x\}$. Then
$$\PP_x(X_{G_{\ee_q}}=k)=\frac{\Phi_q(x)}{\Phi_q(k)}-\frac{\Phi_q(x)}{\Phi_q(k-1)}\mathbbm{1}_{\mathbb{N}}(k).$$ 
Moreover, provided $\PP(X_{G_{\ee_q}}=k)>0$, for further  $\alpha\in [0,\infty)$:
$$ \PP_x[e^{-\alpha G_{\ee_q}}\vert X_{G_{\ee_q}}=k]=\frac{\Phi_{q+\alpha}(x)}{\Phi_{q}(x)}\frac{\Phi_{q}(k)}{\Phi_{q+\alpha}(k)};$$
 hence, when $q>0$, $$\PP_x[e^{-\alpha(\ee_q-G_{ \ee_q})}\vert X_{G_{\ee_q}}=k]=\frac{q}{q+\alpha}\frac{1- \frac{\Phi_{q+\alpha}(k)}{\Phi_{q+\alpha}(k-1)}\mathbbm{1}_{\mathbb{N}}(k)}{1-\frac{\Phi_q(k)}{\Phi_q(k-1)}\mathbbm{1}_{\mathbb{N}}(k)}.$$
\end{corollary}
\begin{proof}
The first statement is immediate from Proposition~\ref{proposition:independnece}. The other statements then follow by routine calculation from Theorem~\ref{corollary:laplace}.
\end{proof}
When $\phi\leq\varphi$, then taking $q=0$ in the preceding corollary provides us with the joint law-Laplace transform of $(X_{G_\infty},G_\infty)$, viz. of the overall infimum of $X$ and the time when this overall infimum is reached. Plainly it is most interesting when the extinction time of $X$ is not a.s. finite (for, when it is, then $G_\infty=T_0^-$ and $X_{G_\infty}=0$ a.s.).

\bibliographystyle{plain}
\bibliography{Branching}

\begin{thebibliography}{10}

\bibitem{asmussen}
S.~Asmussen and H.~Hering.
\newblock {\em Branching Processes}.
\newblock Progress in probability and statistics. Birkh{\"a}user, 1983.

\bibitem{athreya}
K.~B. Athreya and P.~E. Ney.
\newblock {\em Branching Processes}.
\newblock Dover Books on Mathematics. Dover Publications, 2004.

\bibitem{Avram2019}
F.~Avram, P.~Patie, and J.~Wang.
\newblock Purely excessive functions and hitting times of continuous-time
  branching processes.
\newblock {\em Methodology and Computing in Applied Probability},
  21(2):391--399, 2019.

\bibitem{vidmar-avram}
F.~Avram and M.~Vidmar.
\newblock First passage problems for upwards skip-free random walks via the
  scale functions paradigm.
\newblock {\em Advances in Applied Probability}, 51(2):472--495, 2019.

\bibitem{bhattacharya}
R.~N. Bhattacharya and E.~C. Waymire.
\newblock {\em A Basic Course in Probability Theory}.
\newblock Universitext - Springer-Verlag. Springer, 2007.

\bibitem{bgt}
N.~H. Bingham, C.~M. Goldie, and J.~L. Teugels.
\newblock {\em Regular Variation}.
\newblock Encyclopedia of Mathematics and its Applications. Cambridge
  University Press, 1987.

\bibitem{caballero2013}
M.~E. Caballero, J.~L. Pérez~Garmendia, and G.~Uribe~Bravo.
\newblock A {L}amperti-type representation of continuous-state branching
  processes with immigration.
\newblock {\em The Annals of Probability}, 41(3A):1585--1627, 2013.

\bibitem{choi}
M.~C.~H. Choi and P.~Patie.
\newblock Skip-free {M}arkov chains.
\newblock {\em Transactions of the American Mathematical Society},
  371(10):7301--7342, 2019.

\bibitem{chung1967markov}
K.~L. Chung.
\newblock {\em Markov chains with stationary transition probabilities}.
\newblock Grundlehren der mathematischen Wissenschaften. Springer, 1967.

\bibitem{chung2006markov}
K.~L. Chung and J.~B. Walsh.
\newblock {\em Markov Processes, Brownian Motion, and Time Symmetry}.
\newblock Grundlehren der mathematischen Wissenschaften. Springer New York,
  2006.

\bibitem{doney}
R.~A. Doney.
\newblock A note on some results of {S}chuh.
\newblock {\em Journal of Applied Probability}, 21(1):192--196, 1984.

\bibitem{ma}
X.~Duhalde, C.~Foucart, and M.~Ma.
\newblock On the hitting times of continuous-state branching processes with
  immigration.
\newblock {\em Stochastic Processes and their Applications}, 124(12):4182 --
  4201, 2014.

\bibitem{dynkin}
E.~B. Dynkin.
\newblock {\em Markov Processes and Related Problems of Analysis}.
\newblock London Mathematical Society Lecture Note Series. Cambridge University
  Press, 1982.

\bibitem{ethier}
S.~N. Ethier and T.~G. Kurtz.
\newblock {\em Markov Processes: Characterization and Convergence}.
\newblock Wiley Series in Probability and Statistics. Wiley, 2009.

\bibitem{field-theory}
R.~Garcia-Millan, J.~Pausch, B.~Walter, and G.~Pruessner.
\newblock Field-theoretic approach to the universality of branching processes.
\newblock {\em Physical Review E}, 98:062107, 2018.

\bibitem{grey}
D.~R. Grey.
\newblock A note on explosiveness of {M}arkov branching processes.
\newblock {\em Advances in Applied Probability}, 21(1):226--228, 1989.

\bibitem{harris}
T.~E. Harris.
\newblock {\em The Theory of Branching Processes}.
\newblock Dover phoenix editions. Dover Publications, 2002.

\bibitem{homecoming}
J.~F.~C. Kingman.
\newblock Homecomings of {M}arkov processes.
\newblock {\em Advances in Applied Probability}, 5(1):66--102, 1973.

\bibitem{kkr}
A.~Kuznetsov, A.~E. Kyprianou, and V.~Rivero.
\newblock The theory of scale functions for spectrally negative {L\'e}vy
  processes.
\newblock In {\em L{\'e}vy Matters II: Recent Progress in Theory and
  Applications: Fractional L{\'e}vy Fields, and Scale Functions}, pages
  97--186. Springer Berlin Heidelberg, Berlin, Heidelberg, 2013.

\bibitem{maps}
A.~E. Kyprianou and Z.~Palmowski.
\newblock Fluctuations of spectrally negative {M}arkov additive processes.
\newblock In C.~Donati-Martin, M.~{\'E}mery, A.~Rouault, and C.~Stricker,
  editors, {\em S{\'e}minaire de Probabilit{\'e}s XLI}, pages 121--135.
  Springer Berlin Heidelberg, Berlin, Heidelberg, 2008.

\bibitem{lambert}
A.~Lambert.
\newblock Population dynamics and random genealogies.
\newblock {\em Stochastic Models}, 24(sup1):45--163, 2008.

\bibitem{omega}
B.~Li and Z.~Palmowski.
\newblock Fluctuations of {O}mega-killed spectrally negative {L\'e}vy
  processes.
\newblock {\em Stochastic Processes and their Applications}, 128(10):3273 --
  3299, 2018.

\bibitem{li-infinite}
J.~Li and A.~G. Pakes.
\newblock Asymptotic properties of the {M}arkov branching process with
  immigration.
\newblock {\em Journal of Theoretical Probability}, 25:122–143, 2012.

\bibitem{li2010measure}
Z.~Li.
\newblock {\em Measure-Valued Branching {M}arkov Processes}.
\newblock Probability and Its Applications. Springer Berlin Heidelberg, 2010.

\bibitem{millar}
P.~W. Millar.
\newblock {A Path Decomposition for Markov Processes}.
\newblock {\em The Annals of Probability}, 6(2):345 -- 348, 1978.

\bibitem{norris1998markov}
J.~R. Norris.
\newblock {\em Markov Chains}.
\newblock Cambridge Series in Statistical and Probabilistic Mathematics.
  Cambridge University Press, 1998.

\bibitem{parthasarathy}
K.~R. Parthasarathy.
\newblock {\em Probability Measures on Metric Spaces}.
\newblock AMS Chelsea Publishing Series. Acad. Press, 1972.

\bibitem{redner}
S.~Redner.
\newblock {\em A Guide to First-Passage Processes}.
\newblock Cambridge University Press, 2001.

\bibitem{rogers2000diffusions}
L.~C.~G. Rogers and D.~Williams.
\newblock {\em Diffusions, Markov Processes and Martingales: Volume 2, It{\^o}
  Calculus}.
\newblock Cambridge Mathematical Library. Cambridge University Press, 2000.

\bibitem{pssmp}
M.~Vidmar.
\newblock Exit problems for positive self-similar {M}arkov processes with
  one-sided jumps.
\newblock {\em arXiv:1807.00486}, 2018.

\bibitem{vidmar2013fluctuation}
M.~Vidmar.
\newblock Fluctuation theory for upwards skip-free {L\'e}vy chains.
\newblock {\em Risks}, 6(3):no. 102, 2018.

\bibitem{woess}
W.~Woess.
\newblock {\em Denumerable {M}arkov Chains: Generating Functions, Boundary
  Theory, Random Walks on Trees}.
\newblock EMS textbooks in mathematics. European Mathematical Society, 2009.

\end{thebibliography}

\end{document}